\documentclass[10pt]{amsart}

\usepackage{amssymb,graphicx, amsmath, amsthm, MnSymbol, array}
\usepackage{wrapfig}
\usepackage{graphicx}
\graphicspath{ {./images/} }

\usepackage{todonotes}

\usepackage{floatflt}

\usepackage{hyperref}

\usepackage{mathbbol}

\usepackage[font=small]{caption}

\def\N{\mathbb N} 

\def\c{\mathcal}

\def\k{{\kappa}}

\def\E{\c E}
\def\F{\c F}
\def\X{\c X}
\def\M{\c M}
\def\N{\c N}

\def\Y{\c Y}
\def\P{\c P}
\def\Q{\c Q}

\begin{document}

\newtheorem{theorem}{Theorem}[section]
\newtheorem{lemma}[theorem]{Lemma}
\newtheorem{proposition}[theorem]{Proposition}
\newtheorem{corollary}[theorem]{Corollary}
\newtheorem{problem}[theorem]{Problem}
\newtheorem{construction}[theorem]{Construction}

\theoremstyle{definition}
\newtheorem{defi}[theorem]{Definitions}
\newtheorem{definition}[theorem]{Definition}
\newtheorem{notation}[theorem]{Notation}
\newtheorem{remark}[theorem]{Remark}
\newtheorem{example}[theorem]{Example}
\newtheorem{question}[theorem]{Question}
\newtheorem{comment}[theorem]{Comment}
\newtheorem{comments}[theorem]{Comments}

\newtheorem{discussion}[theorem]{Discussion}

\renewcommand{\thedefi}{}

\long\def\alert#1{\smallskip{\hskip\parindent\vrule%
\vbox{\advance\hsize-2\parindent\hrule\smallskip\parindent.4\parindent%
\narrower\noindent#1\smallskip\hrule}\vrule\hfill}\smallskip}

\def\ff{\frak}
\def\tf{torsion-free}
\def\Spec{\mbox{\rm Spec }}
\def\Proj{\mbox{\rm Proj }}
\def\hgt{\mbox{\rm ht }}
\def\type{\mbox{ type}}
\def\Hom{\mbox{ Hom}}
\def\rank{\mbox{ rank}}
\def\Ext{\mbox{ Ext}}
\def\Tor{\mbox{ Tor}}
\def\Ker{\mbox{ Ker }}
\def\Max{\mbox{\rm Max}}
\def\End{\mbox{\rm End}}
\def\xpd{\mbox{\rm xpd}}
\def\Ass{\mbox{\rm Ass}}
\def\emdim{\mbox{\rm emdim}}
\def\epd{\mbox{\rm epd}}
\def\repd{\mbox{\rm rpd}}
\def\ord{\mbox{\rm ord}}

\def\DD{{\mathcal D}}
\def\EE{{\mathcal E}}
\def\FF{{\mathcal F}}
\def\GG{{\mathcal G}}
\def\HH{{\mathcal H}}
\def\II{{\mathcal I}}
\def\LL{{\mathcal L}}
\def\MM{{\mathcal M}}
\def\PP{{\mathcal P}}
\def\P{{\mathbb{P}}}
\def\k{\mathbb{k}}

\def\R{\mathbb{R}}

\def\C{{\bf C}}

\title[Paths of rectangles]{Paths of   rectangles  inscribed in lines over fields}

\author{Bruce Olberding} 
\address{Department of Mathematical Sciences, New Mexico State University, Las Cruces, NM 88003-8001}

\email{bruce@nmsu.edu}

\author{Elaine A.~Walker}
\address{1801 Imperial Ridge, Las Cruces, NM 88011}

\email{miselaineeous@yahoo.com}

\begin{abstract} 
We study   rectangles inscribed in lines in the plane by parametrizing these rectangles in two ways, one involving slope and the other aspect ratio. This produces two paths, one that finds rectangles with specified slope and the other rectangles with specified aspect ratio. We describe the geometry of these paths and its dependence on the choice of four lines. Our methods are algebraic and work over an arbitrary field. 
 \end{abstract}


\thanks{\today}

\maketitle

\section{Introduction} 

Given four lines in the plane that are not all parallel, there is a rectangle whose vertices lie on these lines; i.e., the rectangle is {inscribed} on the lines.
 There are always more inscribed rectangles nearby and all of these rectangles appear as part of a path of inscribed rectangles through the configuration of lines. 
The purpose of this article is to account for all these rectangles by describing the paths they take through the configuration, and to do so by locating them by their slope and aspect ratio. Our methods are algebro-geometric, elementary and work over an arbitrary field.  While it is possible to prove some of the results of the paper computationally, the equations in raw form  are often unwieldy and opaque, so  
we have sought to give algebraic insight into how the geometry of the solution depends on the initial four lines.

The present article continues our study  \cite{OW}  but is mostly independent of this previous work, where the problem of finding inscribed rectangles  was recast as that  of  finding the intersection of  hyperbolically rotated  cones in ${\mathbb{R}}^3$.  Schwartz \cite{Sch} has also recently treated rectangles inscribed in lines in the case in which none of the lines involved are parallel or perpendicular to each other.  
 The indirect motivation for  both Schwartz's work and ours is the so-called square peg problem---a problem that remains open in full generality---of finding a square inscribed in every simple closed curve in the plane.  
By a theorem of Vaughn \cite[p.~71]{Mey}, every simple closed curve in the plane has a rectangle inscribed it. In particular, every polygon has a rectangle inscribed in it, and in fact has a square inscribed in it \cite{Emc}. While Vaughn's proof guarantees that an inscribed rectangle must exist, it does not give a means for finding the rectangle, even when the curve in question is a polygon. To find a rectangle inscribed in a polygon amounts to finding a rectangle inscribed in four possibly non-distnct lines, which brings us to the current problem of finding {\it all} rectangles inscribed in four lines, as well as locating these rectangles by their slopes and aspect ratios.

Schwartz's approach is a mix of computational and  topological methods based on the geometry of the real plane, and so while some of his methods don't extend to arbitrary fields (or {\it because} they don't extend to fields), the difference of his methods when contrasted with ours shows 
some of the richness of the problem. The two approaches often differ  in methods (algebro-geometric properties of fields vs.~topological and analytic properties of the real plane)
   and in the
  directions these results are carried in search of finer-grained consequences. 
We explain in the course of the paper the points of overlap with Schwartz's article, but briefly, Schwartz works over the field of real numbers  and  under the assumption that 
 none of the lines $A,B,C,D$ are parallel or perpendicular.  
  Avoiding  parallel lines is needed to reduce the study of inscribed rectangles to the study of their centers only, and hence to points in $\R^2$. Without this restriction, two inscribed rectangles may share the same center and so information can be lost when considering only rectangle centers. 
  
  We allow parallel and perpendicular lines throughout, and to do so we view inscribed rectangles as points in a four-dimensional projective space. We find two paths of inscribed rectangles in this space, one parameterized by rectangle slope and the other by aspect ratio. Our  point of view permits paths of rectangles to pass through rectangles inscribed at infinity for the four lines.  We study these rectangles in some detail because they help shed light on the rectangles inscribed in the four lines in affine space.
 In the case in which the four lines form a non-degenerate configuration, a case explained shortly  and which is described in Section~8, Schwartz has versions of these parameterizations in affine space and in his setting of no parallel or perpendicular lines. See Section 6 for a discussion of this.    
  
  Some calculations involving the data of the four lines are necessary, but we have tried to frontload these  early in the paper so as to give streamlined and more conceptual proofs later. One of the main goals of the paper is to pinpoint  the essential quantities needed   
 for finding the paths of  inscribed rectangles. These quantities turn out to be the slopes of the diagonals of the configuration.
 
 \begin{figure}[h] \label{diagonalspar}
 \begin{center}
 \includegraphics[width=0.85\textwidth,scale=.09]{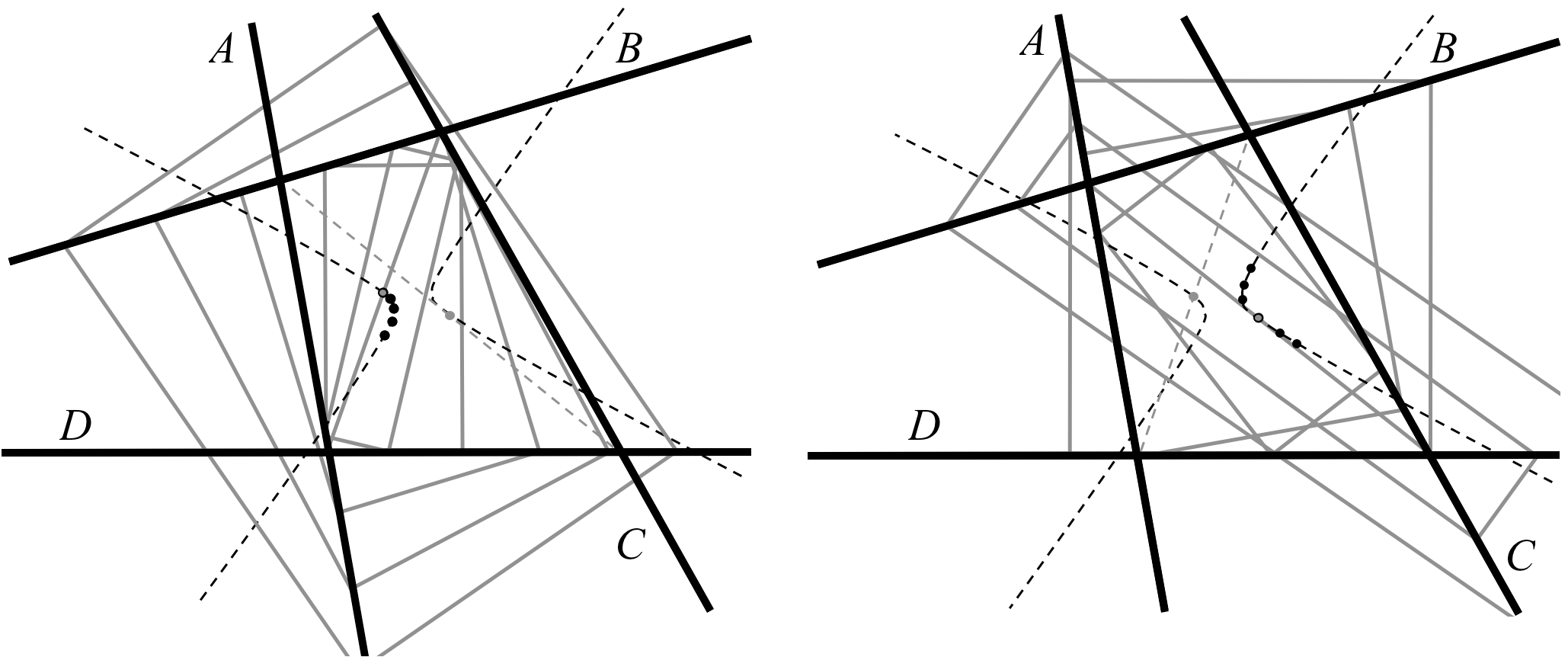} 
 \end{center}
 \caption{Two sets of rectangles inscribed in four lines. The hyperbola  is the set of centers of the inscribed rectangles.  
The fact that the hyperbola here is non-degenerate (because its diagonals are not orthogonal; see Theorem~\ref{deg char}) implies that no two rectangles inscribed in this configuration have the same slope or the same aspect ratio (see Section~8).  
  }
\end{figure}

\begin{figure}[h] \label{degenerateconfig}
 \begin{center}
 \includegraphics[width=0.85\textwidth,scale=.09]{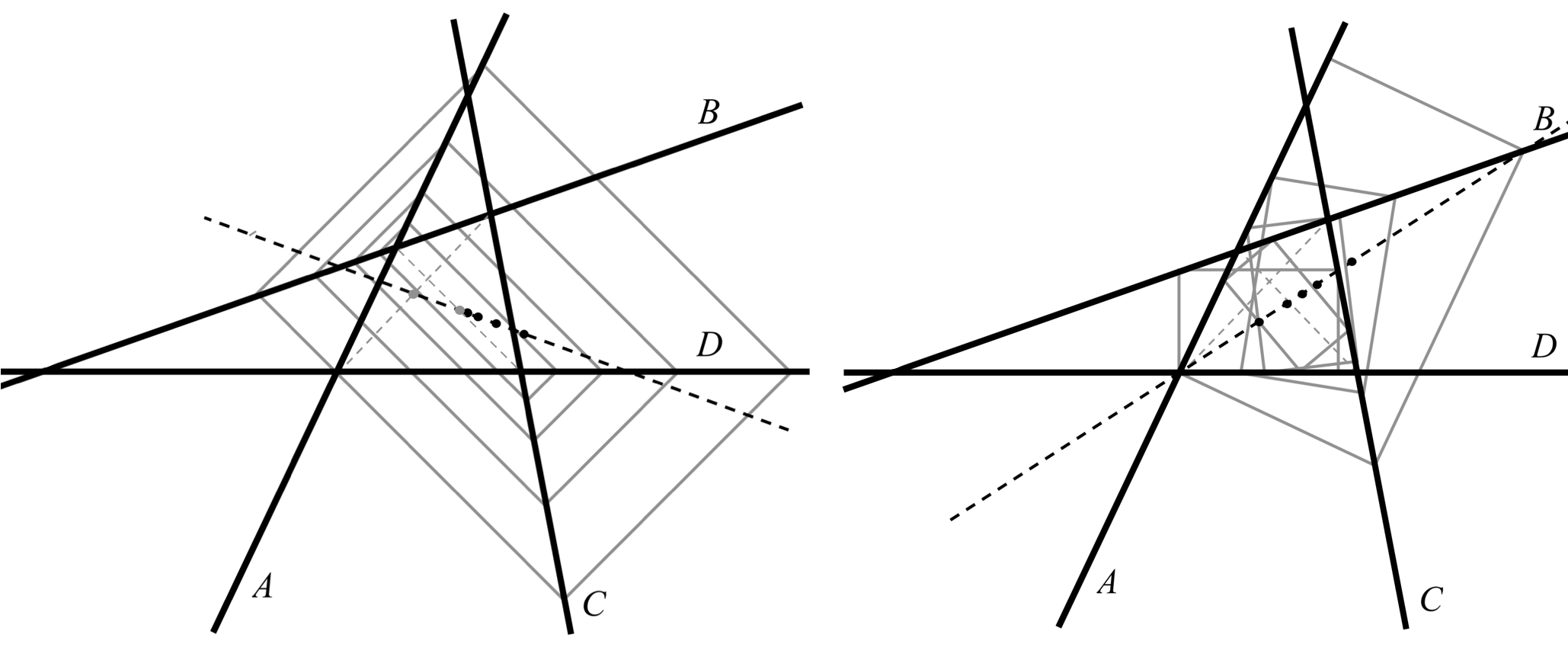} 
 \end{center}
 \caption{This configuration is degenerate because its diagonals are orthogonal (Theorem~\ref{deg char}). The darker dotted line  through each figure is a set of centers of inscribed rectangles. In the first figure, slope remains constant while aspect ratio changes, and in the second aspect ratio remains constant while slope changes. This phenomenon only happens in the degenerate case; compare to Figure 1. The centers of the rectangles in the first figure lie on the Gauss-Newton line through the midpoints of the diagonals, and in the second the centers lie on a line that is parallel to the diagonal through the intersection of $A$ and $C$ and that of $B$ and $D$ (Theorem~\ref{Newton}).   }
\end{figure}

For the purpose of giving  intuition, we briefly describe the behavior of the rectangles inscribed in four lines $A,B,C,D$.  Figure 1 shows two sets of rectangles inscribed on four lines in ${\mathbb{R}}^2$.  These examples suggest a path of rectangles, and it is with the description of this path that the paper is concerned. 
The rectangles whose vertices are inscribed in sequence in the lines $A,B,C,D$ can be viewed as points in the set  $\C = A \times B \times C \times D$. For additional flexibility, we consider not just $\C$ but all scaled copies of $\C$ and their rectangles. The union of these sets is  a five-dimensional vector space and so we may consider its projective space $\P\C$.  
For our purposes, $\P\C$ turns out to be the place to work. Thus we shift 
 focus to  projective space and work out the geometry of the inscribed rectangles there.

We show in Section 3 that the points in $\P\C$ that  represent the inscribed rectangles   comprise a plane curve of degree~$2$. This curve, as we show in Sections~5 and~6, 
 is the union of two paths: a {\it slope path} given by a regular map ${\mathbb{P}}^1 \rightarrow \P\C$ that sends a ratio represented as a point on the projective line  ${\mathbb{P}}^1 $ to  a rectangle having this slope, and an {\it aspect path} given by a regular map ${\mathbb{P}}^1  \rightarrow \P\C$ that sends a ratio   to a rectangle having this aspect ratio. These paths pick up additional rectangles, those that are inscribed at infinity for the configuration, and these rectangles are the subject of Section 4.  In any case, 
  the slope and aspect paths solve the problem of finding the rectangles of specified slope and aspect ratio. In Sections~5 and 6 we give succinct versions of the defining polynomials for these paths in order to exhibit  some of the internal symmetry of the algebra of the paths. The formulations of these polynomials belong  to our main results. The two paths share a number of formal features that suggest more could be learned  about the relationship between them.

 In Sections 7 and 8 we show that the slope and aspect paths either (a) have exactly the same image in $\P\C$, and thus find the same rectangles, or (b) they are distinct lines of rectangles. Case (b) occurs precisely if the diagonals of the configuration  are orthogonal (Theorem~\ref{deg char}). 
In  case (a), the affine piece of the curve of rectangles in $\P\C$ is a non-degenerate conic (Theorem~\ref{deg char}), which if  $\k = {\mathbb{R}}$ is a hyperbola (Theorem~\ref{planar}).  
%
Otherwise, in case (b), each path is a line, so the two lines form a degenerate conic.  (It can happen that one of these lines is the line at infinity.) Case (b) is illustrated by the examples in Figure 2.  
In Section 9 we go further with the degenerate case and interpret the paths of the centers of the rectangles in these two lines geometrically.

 A final word on generality: Even though our main motivation is the case in which $\k$ is the field $\R$ of real numbers, we work with an arbitrary field for two reasons. First, doing so comes at no extra expense since our arguments are algebro-geometric and need no modification to be cast in the general setting of fields. 
So while we have no specific application in mind for, say, rectangles over finite fields, our approach does apply to such rectangles without any additional effort so it seems worthwhile to note this, as also it emphasizes the purely algebraic nature of our arguments. Our approach is therefore similar in spirit to the texts \cite{Kap} and \cite{Sam} of Kaplansky and Samuel. 
 Second, working over a field makes it possible to find rectangles whose vertices are restricted to a subfield of the real numbers. For example, by applying the results of the paper to the field of rational numbers, we find inscribed rectangles in the real plane whose vertices have rational coordinates.   
 
We used  Maple$^{\rm TM}$ to assist with calculations and examples. 








\section{Complete quadrilaterals in the projective plane}

Throughout the paper $\k$ denotes a field. In only a few instances the choice of field matters, and it is only in these cases we put additional hypotheses on $\k$.  
The reductions made in the following standing assumptions, which will be in force for the rest of the article, simplify the presentation and as explained below essentially result in no loss of generality when working with four not necessarily distinct lines $A,B,C,D$ that (a) are not all parallel, (b) do not all go through the same point, and (c) are not vertical. The condition (c) is used to simplify equations and in almost all situations,  reflection 
about a  line through the origin allows us to reduce to this case. For example, this is possible if the field has characteristic other than $2$ and  at least 9 elements. This follows from  
 \cite[Theorem 14, p.~18]{Kap}.

\begin{quote} {\bf Standing assumptions}. 
Throughout this article we work with two pairs of lines $A,C$ and $B,D$ such that $B \ne D$; 
$C$ and $D$  are not parallel;
and there are  constants $m_A,m_B,m_C,m_D,b_A \in \k$ such that the equations defining the lines $A,B,C,D$ are 
$$A: \: y = m_Ax+b_A \:\:\: B: y=m_Bx+1 \:\:\:
C: y=m_Cx \:\:\: D: y=m_Dx.$$


 

 

{\noindent}To help with later notation, we also let $b_B =1$ and $b_C=b_D=0$ so that each line $L \in \{A,B,C,D\}$ is defined by an equation of the form $y=m_Lx+b_L$.
The following quantities, which are interpreted in Lemma~\ref{Delta 0}, will  play a fundamental role in our calculations. We write $m_{AB}$ for $m_A-m_B$, $m_{BC}$ for $m_B - m_C$, etc.. 
\smallskip

\hspace{.45in}
$\begin{array}{ll}
e_1= b_Am_B-m_A \:\:\: & \:\:\: e_2 = b_A-1 \\
f_1 = b_Am_{BC}m_D+m_{DA}m_C \:\:\: & \:\:\: f_2=m_{DA}+b_Am_{BC}
\end{array}$




\end{quote}

\smallskip

In seeking rectangles inscribed in  four lines $A$, $B$, $C$, $D$ satisfying (a), (b) and (c) above, 
we can reduce to the  standing assumptions as follows. 
First, we can 
 relabel lines within the pairs $A,C$ and $B,D$ to guarantee that $C$ and $D$ are not parallel. A translation then guarantees that $C$ and $D$ meet at the origin. 
If  $B$ goes through the origin,  we can switch the labels of $A$ and $C$ with those of $B$ and $D$
 so as to arrange  that $B$ does not pass through the origin. Scaling allows us to assume the $y$-intercept of $B$ is $b_B=1$.  Note also that $A$ can be equal to $B,$ $C$ or $D$, and so the standing assumptions  cover the case of rectangles inscribed in $3$ lines. 
 

%

We work often in projective space. For this, we use standard notation for homogeneous coordinates: If 
 $x_0,\ldots,x_n \in \k$, with not all the $x_i$ equal to $0$, then  $$[x_0:\cdots:x_n] = \{(tx_0,\ldots,tx_n):0\ne t \in \k\}.$$
Projective $n$-space, the collection of all such points $[x_0:\cdots:x_n]$, is denoted $\P^n$.  The {\it points at infinity} for $\P^n$ are the points of the form $[x_0:\cdots:x_{n-1}:0]$, where the $x_i$ are in $\k$ and not all zero. 
 With this in mind, in the next definition we view the configuration of the four lines $A,B,C,D$ as defining axes  on which lie the vertices of the inscribed rectangles. We  allow for scaled copies of these axes, something that will be important for defining rectangles in  four-dimensional projective space.

   \begin{definition} \label{first notation}

For each line $L \in \{A,B,C,D\}$ and $w \in \k$, let 
$L_w$ be  the line in $\k^2$ whose equation is $y=m_Lx+b_Lw$. Let 
$${\bf C}_w=A_w \times B_w \times C_w \times D_w, \:\: \: \C=\C_1 = A \times B \times C \times D. $$
Define a subspace  $\P\C$ of  $\P^8$   by 
$$\P\C = \left\{[x_A:y_A:\cdots:x_D:y_D:w] \in \P^8:(x_A,y_A,\ldots,x_D,y_D) \in \C_w\right\}.$$
Alternatively, $\P\C$ can be viewed as the 
the projective space of the five-dimensional   subspace $\bigcup_{w \in \k}\C_w$ of   $\k^8$. 
The {\it points at infinity} for $\P\C$ are the points in $\P\C$ with $w=0$.  


\end{definition}

The special case $\C_0$ (i.e., $\C_w$ where $w=0$) is a configuration of the four lines $A_0,B_0,C_0,D_0$ through the origin. These is the  configuration $\C$ viewed from infinity, and it plays an important role in analyzing rectangles at infinity for $\C$.  


In the next section we will interpret parallelograms and rectangles in the plane as points in $\P\C$, a flexibility that allows us to define rectangles at infinity also. To do so, it is helpful to view the lines $A,B,C,D$ as lines in the projective plane that form a 
 {\it complete quadrilateral} that consists of four lines and six points of intersection, where some of these six points are possibly at infinity. 
These
  six points  give rise to three diagonals.  (Because we work over a field, the notion of a segment of a line may not have meaning, so we view our diagonals as lines rather than segments.)

\begin{figure}[h] \label{diagonals}
 \begin{center}
 \includegraphics[width=0.85\textwidth,scale=.09]{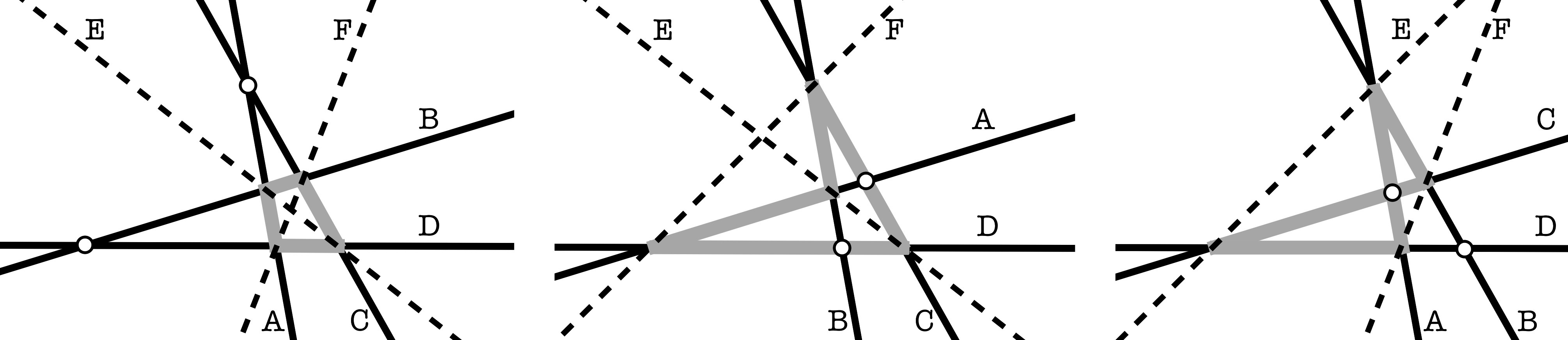} 
 \end{center}
 \caption{Three complete quadrilaterals $ABCD$ and their diagonals $E$ and $F$. The circles mark the pairs $A,C$ and $B,D$ that determine the configuration for each complete quadrilateral. Since none of the lines are parallel, each configuration determines a different quadrilateral (shaded gray) whose sides lie in sequence on $A,B,C,D$.  
 }
\end{figure}


\begin{definition} 
We define the {\it diagonals} $E$ and $F$ for the configuration $\C$ as follows. 
\begin{enumerate}
\item   If $A \ne B$, then 
 $E$ is the line in $\P^2$ through the origin (the intersection of $C$ and $D$) and the intersection of $A$ and $B$, which is possibly at infinity.  Otherwise, if $A = B$, then $E$ is the line through the origin that is orthogonal to $A$, i.e., the line defined by $x+m_Ay=0$.

\item  
If  $A \ne D$, then 
$F$ is the line in $\P^2$ through the intersection  of $A$ and $D$ and that of $B$ and $C$. 
Otherwise, if $A = D$, then $F$ is the line through the intersection of $B$ and $C$ that is orthogonal to $A$. 


 
\end{enumerate}

\end{definition}

 The third diagonal for the configuration can be defined similarly but only a special case of it will be needed, and not until Section 9, so we define it there. 
 Figure 3 exhibits some complete quadrilaterals and their diagonals in the  case in which no two lines are parallel. The case in which some of the  lines are parallel is clarified by the following remark and Figure 4.

\begin{remark} \label{diagonal remark} $\:$

\begin{enumerate}

\item 
If $A \parallel B$ and $A \ne B$, then $E$ is the line through the intersection of $C$ and $D$ that is parallel to $A$ and $B$; see the first configuration in Figure~4.  If $A = B$, then $F$ coincides with $A$ and $B$ as in the second configuration, and $E$ is an altitude for the triangle that is formed from the four lines in $\C$.  

\item If   $A \parallel D$ and $B \parallel C$, then $F$ is the line at infinity; see the third configuration in Figure~4.
Otherwise, if $A$ and $D$ intersect in a single point and $B$ and $C$ are parallel, then $F$ is the line  parallel to $B$ and $C$  through this point.  
Similarly, if $B$ and $C$ intersect in a single point and $A $ and $D$ are parallel and distinct, then $F$ is the line through the point that is parallel to $A$ and $D$.

\end{enumerate} 
\end{remark} 

\begin{figure}[h] \label{diagonalsparallel}
 \begin{center}
 \includegraphics[width=0.85\textwidth,scale=.09]{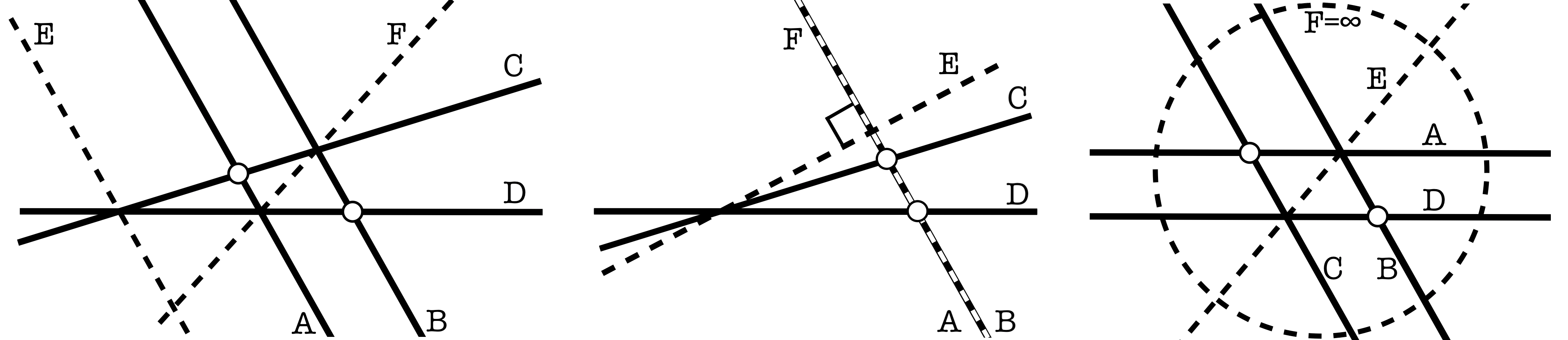} 
 \end{center}
 \caption{The diagonals of three complete quadrilaterals $ABCD$ with at least one set of parallel lines.
 In the second figure, the lines $A$ and $B$ are equal, and so $F$ is this same line also. 
  In the third figure, the diagonal $F$ is the line at infinity since the diagonal passes through the intersection of $A$ and $D$ and the intersection of $B$ and $C$.  
 }
\end{figure}

We end this section with a technical lemma that explains  the significance of the  constants $e_1,e_2,f_1,f_2$ from the standing assumptions.  The purpose of the  lemma  and 
the formulation of the constants $e_1,e_2,f_1,f_2$  
 is to shift what would be tedious case-by-case calculations throughout the paper to a few such calculations here in this proof.   
The lemma refers to the slope of a line, which will have the usual meaning, but which we often interpret as a point on the projective line $\P^1$ so as to have some flexibility with infinite slope. 
Specifically, we view  the {slope} of a line through two distinct points $(x_1,y_1)$ and $(x_2,y_2)$ in $\k^2$ as the point $[y_1-y_2:x_1-x_2] \in \P^1$, and we write this equivalence class as $(y_1-y_2)/(x_1-x_2)$, where possibly $x_1-x_2=0$.  We continue also to represent slopes such as $m_A,m_B,$ etc.,  as an element $m$ of $\k$ but sometimes interpret $m$ as the point $m/1 \in \P^1$.

\begin{lemma}
\label{Delta 0}
$A=B$ if and only if $e_1=e_2=0$. Otherwise, if $A \ne B$, the slope of $E$ is  $e_1/e_2$. 
Similarly, $A = D$ or $F$ is the line at infinity  if and only if  $f_1=f_2=0$. Otherwise, 
the slope of $F$ is $f_1/f_2$.

\end{lemma} 

\begin{proof}
If $e_1=e_2=0$, then $b_Am_B -m_A =0$ and $b_A = 1$, so  $m_B = m_A$ and hence  $A =B$.  The converse is clear.
Now suppose that~$A \ne B$.  If $A \parallel B$, then by Remark~\ref{diagonal remark},
 $m_E= m_A = m_B$. Since $A \ne B$, this implies that $b_A \ne 1$.  Thus $[m_E:1] = [m_A(b_A-1):b_A-1],$ 
 so that $E$ has slope $ e_1/e_2$.  Otherwise, if $A$ is not parallel to $B$, then $A$ and $B$ intersect in the point   $$\left(\frac{1-b_A}{m_{AB}}, m_A\left(\frac{1-b_A}{m_{AB}}\right)+b_A\right) = 
\left(\frac{1-b_A}{m_{AB}}, \frac{m_A-m_Bb_A}{m_{AB}}\right).$$
Since $C$ and $D$ intersect at the origin,   the slope of $E$ is $(m_A-m_Bb_A)/(1-b_A)=e_1/e_2$.

We show next that $f_1=f_2=0$ iff $A = D$  or  $A \parallel D$ and  $B \parallel C$.
Suppose that $f_1=f_2=0$. Then  $m_{DA} + b_Am_{BC}=f_2 =0$ yields  
 $$0=f_1= b_Am_{BC} m_D +m_{DA}m_C = -m_{DA}m_D + m_{DA}m_C = m_{DA}m_{CD}.$$  Since $C$ is not parallel to $D$, it must be that 
 $m_A = m_D$. Thus  $ 0 = m_{DA} + b_Am_{BC}  =b_Am_{BC}$. 
 If $b_A \ne 0$, then
 $m_B = m_C$ and $A \parallel D$ and $B \parallel C$;   
otherwise, if $b_A =0$, then $A = D$.    This proves that if $f_1=f_2=0$, then $A=D$  or $F$ is the line at infinity. 
 The converse is straightforward.
 
 It remains to prove the last assertion. 
Suppose that $A \ne D$, $A \parallel D$ and $B$ is not parallel to $C$. Then $b_A \ne 0$ and   $m_B \ne m_C$.   Also, $m_F = m_A = m_D$.  Therefore,  \begin{eqnarray*}
[m_F:1] & = & [m_D:1] \: = \: [b_Am_{BC}m_D:b_Am_{BC}] \\
& = & 
[b_Am_{BC}m_D+m_{DA}m_C:m_{DA}+b_Am_{BC}]\\
& = & [f_1:f_2].  
\end{eqnarray*}
Thus the slope of $F$ is  $f_1/f_2$.  A similar argument shows that if $B \parallel C$ and $A$ is not parallel to $ D$, then the slope of $F$ is again $f_1/f_2$.  

Finally, suppose that $A$ is not parallel to $D$ and $B$ is not parallel to $C$.   The two points   $$
\left(\frac{b_A}{m_{DA}}, \frac{m_Db_A}{m_{DA}}   \right), \: \: 
\left(\frac{1}{m_{CB}},\frac{m_C}{m_{CB}}  \right) 
$$
are the points of intersection of $A$ and $D$ and $B$ and $C$, respectively. The slope  of the line $F$ through these two points is $$\left[ {m_{DA}m_C-m_{CB}m_Db_A}: {m_{DA}-m_{CB}b_A} \right] \: = \: [f_1:f_2],$$
which proves the lemma. 
\end{proof}

\section{The curve of inscribed rectangles}

We define parallelograms and rectangles inscribed in the lines $A,B,C,D$ as 
points in $\P\C$.   
  

\begin{definition} Let  $w \in \k$. A  {\it parallelogram in ${\bf C}_w$} is a
point $(x_A,y_A,\ldots,x_D,y_D)\in \C_w$ such that $x_A-x_B = x_D-x_C$ and $y_A-y_B = y_D -y_C$. The points $(x_A,y_A) ,\cdots,(x_D,y_D)$ are the {\it vertices} of the {parallelogram}. 
A {\it parallelogram in $\P{\bf C}$} is a
point $[x_A:y_A:\cdots:x_D:y_D:w] \in \P\C$ such that $(x_A,y_A,\ldots,x_D,y_D)$ is a parallelogram in $ \C_w$. The parallelogram 
is {\it at infinity} if $w =0$.  
 We allow the possibility that two or more of the vertices 
are the same point, and in this case, we say that the parallelogram is {\it degenerate}. \end{definition} 

Thus the parallelograms inscribed in sequence\footnote{A parallelogram $[x_A:y_A:\ldots:x_D:y_D:w]$ in $\P\C$ has vertices inscribed in sequence in the lines $A_w,B_w,C_w,D_w$.   
If we wish to find parallelograms or rectangles whose vertices fall   in a different sequence, we would define the Cartesian product $\C$ using the lines in a different sequence.  
As discussed in \cite[Section 4]{OW}, finding all rectangles inscribed on four lines requires finding all the rectangles inscribed in  21 configurations involving these four lines. (In some of these configurations, two pairs of lines share a line; these configurations are needed to find the rectangles having two vertices on the same line.) In any case, the important point here is:
 {\it Finding rectangles inscribed in lines reduces  to finding the  rectangles in $\C$.}}
 in the lines $A,B,C,D$ are the parallelograms in $\C_1$. Those inscribed in scaled copies $A_w,B_w,C_w,D_w$ of $A,B,C,D$ are  the parallelograms in $\C_w$.  A parallelogram in $\P\C$ is an equivalence class of the scaled copies of a parallelogram in $\C$ or $\C_0$, the latter being the parallelograms at infinity.   


Rectangles are parallelograms whose vertices are subject to an additional condition:

\begin{definition} 
\label{rectangle def} 
Let  $w \in \k$. A  {\it rectangle} 
is a
parallelogram $(x_A,y_A,\ldots,x_D,y_D)\in \C_w$ such that $(x_C-x_B)(x_{B}-x_A) + (y_{C}-y_B)(y_{B}-y_A)=0.$
A {\it rectangle  in $\P{\bf C}$} is a 
parallelogram 
$[x_A:y_A:\ldots:x_D:y_D:w]$
such that   $(x_A,y_A,\ldots,x_D,y_D)$ is a rectangle in $ \C_w$. The rectangle is {\it at infinity} if $w=0$. See Figure 5 for examples of rectangles at infinity. 
 \end{definition}
 
 \begin{figure}[h] \label{diagonalsparallel}
 \begin{center}
 \includegraphics[width=0.95\textwidth,scale=.09]{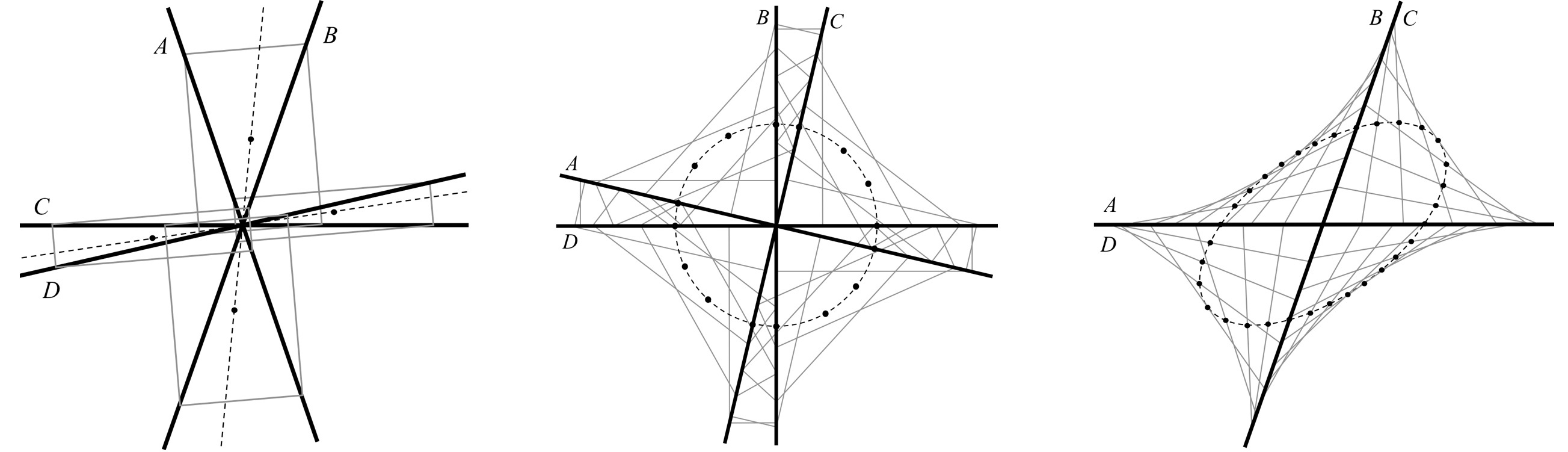} 
 \end{center}
 \caption{Rectangles at infinity. The first figure shows the two rectangles at infinity and their reflections  as represented in $\C_0$ when  none of the lines are parallel. 
 The other two figures represent twin pairs. In these cases, there are rectangles at infinity of every possible slope; see Corollary~\ref{ap}.
In the third figure, each rectangle is degenerate. In fact, these are precisely the degenerate rectangles of unit length inscribed in $\C_0$. The centers of such degenerate rectangles define an ellipse. 
 }
\end{figure}

 If ${\k}$ is the field of real numbers, then 
the  equation  in Definition~\ref{rectangle def} implies that  
  the line passing through the parallelogram's vertices on lines $A$ and $B$ is perpendicular to the line passing through the vertices on $B$ and $C$.  Interpreting this condition as an orthogonality condition for  fields that are not formally real can be problematic: If $\k$ is the field of complex numbers, then the same line  can be ``orthogonal'' to itself under this definition (e.g., the line $y = ix$ has this property). Thus for fields such as the field of  complex numbers, what we are calling  rectangles may not match with other natural notions of rectangles defined using inner products more typical for the choice of such a field.\footnote{For simplicity's sake, in this article we work only in the case in which the inner product is the dot product. In a future paper, when we need to work over an inner product space,  we will view 
  the inner product space as a linear transformation of a vector space equipped with the dot product and in doing so apply the results of the current paper.} But our primary interest is in formally real fields,
 including  ${\mathbb{R}}$ itself, and in these cases, our algebraic rectangles reflect an obvious choice of orthogonality relation.  


\begin{theorem}  \label{rectangle in}
The set of parallelograms in $\P\C$ is a plane, and the set of rectangles  is a  curve of degree $2$   in this plane. The line at infinity for this plane is the set of parallelograms at infinity, while the points at infinity for the curve of rectangles in the plane are the rectangles at infinity. 
\end{theorem} 

\begin{proof} To prove that the set of parallelograms in $\P\C$ is a plane, it suffices to show that  
a  point $p=[x_A:y_A:\cdots:x_D:y_D:w] \in \P{\bf C}$ is a parallelogram  if and only if 
\begin{eqnarray} \label{xc}
 x_C =\frac{1}{m_{DC}} \cdot
\left(m_{AD}x_A+m_{DB}x_B+(b_A-1)w\right) {\mbox{ and }}   x_D = x_A-x_B+x_C.
\end{eqnarray}
 Suppose that $p$ is a parallelogram. Using the fact that the vertices of $p$ lie on the lines $A_w,B_w,C_w,D_w$, we have   \begin{center} $x_A-x_B = x_D-x_C$ \: and \: $m_Ax_A-m_Bx_B+(b_A-1)w = m_Dx_D-m_Cx_C.$ 
\end{center}
Rewriting,   
$$\begin{bmatrix} 1 & -1 \\
m_A & -m_B \\
\end{bmatrix}
\begin{bmatrix}
x_A \\ 
x_B\\
\end{bmatrix}
+
\begin{bmatrix} 
0 \\ (b_A-1)w \\
\end{bmatrix} 
= \begin{bmatrix} 1 & -1 \\
m_D & -m_C \\
\end{bmatrix}
\begin{bmatrix}
x_D \\ 
x_C\\
\end{bmatrix}. $$
Now $m_C \ne m_D$ since the lines $C$ and $D$ are not parallel, so \begin{eqnarray*}
\begin{bmatrix}
x_D \\ 
x_C\\
\end{bmatrix} & = & 
\frac{1}{m_{DC}}
\begin{bmatrix}
-m_C & 1 \\
-m_D & 1 \\
\end{bmatrix}
\left(
\begin{bmatrix} 1 & -1 \\
m_A & -m_B \\
\end{bmatrix}
\begin{bmatrix}
x_A \\ 
x_B\\
\end{bmatrix}
+
\begin{bmatrix} 
0 \\ (b_A-1)w \\
\end{bmatrix} 
\right) \\
& = & 
\frac{1}{m_{DC}}
\left(
\begin{bmatrix}
m_{AC} & m_{CB} \\
m_{AD} & m_{DB} \\
\end{bmatrix}
\begin{bmatrix}
x_A \\ 
x_B\\
\end{bmatrix}
+
\begin{bmatrix} 
(b_A-1)w\\ (b_A-1)w \\
\end{bmatrix} 
\right) \\
\end{eqnarray*}
Thus $x_C$ is as claimed. Since $x_A-x_B = x_D -x_C$, we have also that $x_D = x_C+x_A-x_B$. 

Conversely, if $x_A,x_B,x_C,x_D,w \in \k$    satisfy the equations in $(\ref{xc})$, then   the above matrix calculations show $$y_A-y_B=m_Ax_A-m_Bx_B+(b_A-1) w= m_Dx_D-m_Cx_C+w=y_D-y_C,$$ so 
$p$  defines a parallelogram in~$\P{\bf C}$. 
%

Now let  $X_A,X_B,X$ be indeterminates for $\k$. Define a polynomial $f$ by 
\vspace{.05in}
\begin{center}$f(X_A,X_B,X) =  \frac{1}{m_{DC}} \cdot
\left(m_{AD}X_A+m_{DB}X_B+(b_A-1)X\right).$
\end{center}
\vspace{.05in}
Applying the relevant definitions and using $(\ref{xc})$, the rectangles in $\P{\bf C}$ are the parallelograms in $\P\C$ of the form  $[x_A:y_A:\cdots:x_D:y_D:w],$ where $(x_A,x_B,w)$
are the zeroes of the polynomial 
\vspace{.05in}
\begin{center} $
h(X,X_A,X_B) \: = \: 
(m_B X_B- m_AX_A +(1-b_A)X)(m_Cf(X_A,X_B,X)-m_BX_B-b_{B}X)$ \\
$ \ \ \ \ + \: (X_B-X_A)(f(X_A,X_B,X)-X_B).$
\end{center}
It follows that the set of rectangles in $\P\C$ is a curve of degree $2$ in the plane of parallelograms in $\P\C$.  
The points at infinity for this plane are the points in $\P\C$ of the form $[x_A:y_A:\cdots:x_D:y_D:~0]$, where $(x_A,y_A,\ldots,x_D,y_D) $ is a parallelogram in $\C_0$. Thus the points at infinity in the plane of parallelograms are the {parallelograms at infinity}, and similarly the rectangles at infinity are the points at infinity for the curve of rectangles. 
%
%
 %
 %
\end{proof}

As we see in the next lemma, outside of one special case there are at most two rectangles at infinity.  
 We discuss rectangles at infinity in more detail in the next section. For now, we need them in the proof of Theorem~\ref{planar} in  order to better identify the nature of the  curve of rectangles in $\C$. In order to  formulate Lemma~\ref{at infinity},  we give names in Definition~\ref{twin} to two 
types of configuration of the lines $A,B,C,D$ that appear in the lemma and recur often in this article as  exceptional cases. 
We say that  a line $L_1$ with slope $s_1/t_1$ is  {\it orthogonal} to a line $L_2$ with slope $s_2/t_2$ (written $L_1 \perp L_2$) if $s_1s_2 + t_1t_2 =0$.
 We adopt the convention that the line at infinity for $\k^2$ is orthogonal to every line in $\k^2$. 

 \begin{definition} \label{twin}  $\:$
 \begin{enumerate}  
 \item $\C$ has {\it  twin pairs} if $A \parallel D$ and $B \parallel C$ or $A \perp C$ and $B \perp D$. 
 \item $\C$ has {\it dual pairs} if  $A \perp A$ and either $A \parallel B \parallel C$  or $A \parallel B \parallel D$.
 \end{enumerate}
 \end{definition}
 
 If $\C$ has dual pairs, then $m_A^2 =-1$, and so the condition of dual pairs cannot occur if  $\k$ is formally real. 
 In the case $\k=\R$,  
 an elliptic cone with apex in the $xy$-plane can be associated to each of the pairs $A,C$ and $B,D$  as in \cite{OW}. The configuration $\C$ has twin pairs if and only if these cones are identical, unrotated copies of each other, the only difference being that their apexes are in different locations. See Proposition 4.3 of \cite{OW} and its proof.

\begin{lemma} \label{at infinity} Every parallelogram at infinity in $\P\C$ is a rectangle at infinity
if and only if  $\C$ has twin pairs or dual pairs.
If $\C$ does not have twin pairs or dual pairs, then 
there are at most two rectangles at infinity in $\P\C$, and 
if $\k=\R$,  there are exactly two. 
\end{lemma}

\begin{proof} With polynomials $f$ and $h$ as in the proof of Theorem~\ref{rectangle in},
 the rectangles at infinity are the parallelograms $[x_A:y_A:\cdots:x_D:y_D:0]$ in $\P\C$ for which   $h(x_A,x_B,0)=0$. Every parallelogram at infinity in $\P\C$ is a rectangle at infinity  if and only if $h(X_A,X_B,0)$ is identically $0$. 
 To prove the first assertion of the lemma, we show that $h(X_A,X_B,0)=0$ if and only if $\C$ has twin pairs or dual pairs.

   Expanding $h(X_A,X_B,0)$, a calculation shows  $$m_{CD}h(X_A,X_B,0) = m_{AD}(m_Am_C+1)X_A^2 -\delta X_AX_B+  m_{BC} (m_Bm_D+1)X_B^2,$$ where
   $\delta= 
   m_{BD}(m_Am_C+1)+m_{AC}(m_Bm_D+1).$
    Thus $h(X_A,X_B,0)=0$ if and only if
    \begin{itemize}
    \item[(i)] $m_A = m_D$ or $m_Am_C=-1$; 
       \item[(ii)] $ m_B=m_C$ or  $m_Bm_D=-1$; and 
    \item[(iii)] $m_{BD}(m_Am_C+1)=-m_{AC}(m_Bm_D+1)$. 
    \end{itemize}
    If $\C$ has twin pairs or dual pairs, then (i), (ii) and (iii) clearly hold. Conversely, if 
     both pairs $A,C$ and $B,D$ are orthogonal, then $\C$ has twin pairs, so 
    suppose  that     at least one of the pairs  is not orthogonal, say $A$ and $C$ are not orthogonal. By (i),  
  $m_A = m_D$. If $m_B = m_C$, then $\C$ has twin pairs, so suppose that $m_B \ne m_C$.  By (ii), $m_Bm_D =-1$, and by (iii), $m_B=m_D$.  Thus $\C$ has dual pairs. 
 On the other hand, if $B$ and $D$ are not orthogonal, then $m_B=m_C$.  If $m_A = m_D$, then $\C$ has twin pairs. Otherwise, if $m_A \ne m_D$, then  by (i), $m_Am_C=-1$, and so by (iii), $m_A=m_C$, so that  $\C$ has dual pairs.  This proves the first assertion of the lemma. 
 


   Now suppose that   $\C$ does not have twin pairs or dual pairs.  
   Then $h(X_A,X_B,0)$ is not identically zero, so 
  this polynomial is homogeneous of degree~$2$. Therefore, there are most two  zeroes of $h(X_A,X_B,0)$ in $\P^1$. Suppose that $\k = \R$. We show that are exactly two zeroes in $\P^1$ for this polynomial.
    A calculation shows the discriminant of $h(X_A,X_B,0)$ is $$\Delta=4(m_Am_C-m_Bm_D)^2+(m_Am_Bm_{CD}+m_Cm_Dm_{AB}-m_{AB}-m_{CD})^2.$$
    We claim that if    $\Delta =0$, then $\C$ has twin pairs. Suppose $\Delta=0$. Then $m_Am_C=m_Bm_D$, and since all four lines $A,B,C,D$ are not parallel, at least one of these slopes is nonzero, say $m_A \ne 0$.  Thus $m_C = \frac{m_Bm_D}{m_A}$, and so  from $m_Am_Bm_{CD}+m_Cm_Dm_{AB}-m_{AB}-m_{CD}=0,$ we obtain 
   $ (m_Bm_D+1)(m_A-m_D)(m_A-m_B)=0.$
   This along with $m_Am_C=m_Bm_D$ and $m_C \ne m_D$ implies that $\C$ has twin pairs, a contradiction. 
   Similarly, if $m_C \ne 0$, then another calculation shows
   that $ (m_C-m_D)(m_Bm_D+1)(m_B-m_C) =0$, and again $\C$ has twin pairs. Symmetric arguments involving $m_B$ and $m_D$ instead of $m_A$ and $m_C$ also yield that $\C$ has twin pairs. 
 Therefore, if $\C$ does not have twin pairs, then $\Delta>0$,   
and $h(X_A,X_B,0)$ has two zeroes in $\P^1$.  Consequently, there are two rectangles at infinity for $\C$.  
\end{proof}

If  $\k=\R$, then 
 the set of rectangles in $\P\C$ is a conic by  Theorem~\ref{rectangle in}. 
The nature of this  conic when restricted to $\C$ is clarified by the next theorem.

\begin{theorem} \label{planar} If $\k = \R$, then the set of rectangles in $\C$ is a line if and only if $\C$ has twin pairs. Otherwise, the set of rectangles in $\C$ is 
 a planar hyperbola. 

\end{theorem}

\begin{proof} 
By Theorem~\ref{rectangle in}, the set of rectangles in $\P\C$ is a conic in the plane of parallelograms in $\P\C$.  
Suppose first that $\C$ has twin pairs. By Lemma~\ref{at infinity}, the line at infinity for the plane of parallelograms in $\P\C$ consists of rectangles at infinity. 
Since $\C$ has twin pairs, we have either that $A \parallel D$ and $B \parallel C$ or
$A \perp C$ and $B \perp D$. In either case,  if $A \parallel B$, then 
 $C$ is parallel to $D$, which is a  contradiction to our standing assumptions.
Therefore, $A$ and $B$ are not parallel.   
Let $(x_1,y_1)$ be  
the intersection of $A$ and $B$, and recall that $(0,0)$ is the intersection of $C$ and $D$. The point $[x_1:y_1:x_1:y_1:0:0:0:0:1]$ defines a degenerate rectangle in $\P\C$ that is not on the line at infinity for the plane of parallelograms. Since the set of rectangles in $\P\C$ is a conic and contains the line at infinity and a point not on this line, it follows that the set of rectangles in $\P\C$ is a degenerate conic that is a pair of distinct lines, one of which is at infinity. Thus the set of rectangles in $\C$ is a line.  

Conversely, if the set of rectangles in $\C$ is a line $L$, then the planar conic that is the set of rectangles in $\P\C$ is degenerate. 
Since by Lemma~\ref{at infinity} there are at least two rectangles at infinity for 
the conic consisting of all the rectangles in $\P\C$, there is a point at infinity on the conic that is not a point at infinity for $L$. Since the conic is degenerate, the point must lie on a  line $L'$ in $\P\C$ that does not contain $L$. 
Since the set of rectangles in $\C$ is $L$, it follows that $L'$ is the line at infinity 
 in the plane of parallelograms. Thus every parallelogram at infinity is a rectangle at infinity. By Lemma~\ref{at infinity}, $\C$ has twin pairs. 

Finally, if the set of rectangles at infinity is not the line at infinity, then by Lemma~\ref{at infinity} there are two points at infinity for the planar conic consisting of the rectangles in $\P\C$.  This conic is therefore  a planar hyperbola in $\C$. 
\end{proof}

Theorem~\ref{planar} implies that if $\k = \R$ and $\C$ has twin pairs,  then the set of rectangles in $\C$ is a line. The proof shows that this line can be viewed as part of a degenerate hyperbola, with the other line occurring at infinity in $\P\C$. 
See Section 9 for a discussion of how the theorem can be used to recover results from \cite{OW} and \cite{Sch} regarding the locus of centers of the rectangles in $\C$.

\section{Rectangles at infinity}

We defined a rectangle in $\P\C$
as an equivalence class   $[x_A:y_A:\cdots:x_D:y_D:w]$   consisting of the uniformly scaled copies of a rectangle $(x_A,y_A,\ldots,x_D,y_D) \in \C_w$.
 Slope and aspect ratio, to be defined below, are invariant for the  rectangles in the equivalence class. In this section we determine the slopes and aspect ratios of the rectangles at infinity. 
 

We   treat slope and aspect ratio of   rectangles in the same way we treated slope in Section~2, namely as points $[a:b]$ in $\P^1$ that we write as $a/b$. 
The following definition is motivated by the idea that the slope of a rectangle  in ${\bf C}$ is the slope of the line through the vertices of the rectangle that lie on $A$ and $B$. However, since we permit degenerate rectangles, these two vertices may coincide, and so we define the slope in this case to be the slope of a line that is orthogonal to the line through the vertices that lie on $B$ and $C$. This last complication is the reason why two equations rather than one are used in the following definition.

  \begin{definition} \label{slope def} 
  A   rectangle  
  $[x_A:y_A:\cdots:x_D:y_D:w] \in \P\C$ 
   has {\it slope} $s/t \in {\mathbb{P}}^1$ if $s$ and $t$ are not both $0$ and are a solution to the system of equations 
\begin{eqnarray} \label{slope eqs}
\begin{cases} 
(x_B-x_A)S - (y_B-y_A)T = 0 \\ 
(y_C-y_B)S+(x_C-x_B)T = 0 \\
\end{cases} 
\end{eqnarray}
\end{definition}

%
The slope of the rectangle is whichever of  $(y_A-y_B)/(x_A-x_B)$ and $(x_B-x_C)/(y_C-y_B)$ is defined. When both are defined, these two points are equal since the line through $(x_A,y_A)$ and $(x_B,y_B)$ is orthogonal to the line through $(x_B,y_B)$ and $(x_C,y_C)$.

We will show that the slopes of the rectangles at infinity are determined by the following polynomial.

\begin{notation} \label{basic poly 2} 
Define a polynomial $\sigma(S,T) \in \k[S,T]$  by  \begin{eqnarray*}
\sigma(S,T) & = &  (m_Am_C-m_Bm_D)S^2 - \beta ST- (m_Am_C-m_Bm_D)T^2.
\end{eqnarray*} 
where $\beta = 
(m_Am_C+1)(m_B+m_D)-(m_Bm_D+1)(m_A+m_C) =
(m_Bm_C-1)(m_A-m_D)+(m_Am_D-1)(m_C-m_B)$

\end{notation}

 Using the arguments such as in  Lemma~\ref{at infinity}, it is not hard to see that  $\sigma$ is identically zero if and only if  $\C$ has twin pairs.

\begin{remark} \label{slope remark} Direct calculation shows that \begin{eqnarray*}  
(S^2+T^2) \cdot \sigma(S,T)  &  = &  (T+m_AS)(S-m_BT)(T + m_C S)(S-m_DT) \\
& & \ \ \ \ \ \ \ \ \ 
- \ \ \ 
 (S-m_AT)(T+m_BS)(S-m_CT)(T+m_DS).
\end{eqnarray*}
\end{remark}


%



\begin{lemma}   \label{four eq a} 
Let $s,t\in\k$, where not both $s,t$ are zero. There  is a $2\times 1$ vector $U$   and a $2\times 2$ matrix $M$
with
 entries in $\k$ and determinant $\sigma(s,t)$  such that for each $w \in \k$,  a parallelogram 
 $[x_A:y_A:\cdots: x_D:y_D: w] \in \P\C$
  is a rectangle of slope $s/t$   if and only if 
 $$M  \begin{bmatrix}
 x_A \\ x_B
 \end{bmatrix} = wU.$$  %
 %

\end{lemma}

\begin{proof}  %
Let $s,t \in \k$ with both $s$ and $t$ not zero, and let $p=[x_A:y_A:\cdots: x_D:y_D: w]$ be a parallelogram in  $\P\C$. 
As in the proof of Theorem~\ref{rectangle in}, 
 $$x_C = \frac{1}{m_{DC}} (m_{AD}x_A+ 
m_{DB}x_B+(b_A-1)w), \:\: x_D = x_A+x_B-x_C.
$$  Since  $p \in \C_w$, 
we have $y_L = m_Lx_L+b_Lw$ for all $L \in\{A,B,C,D\}$.  Let 
$$M=\begin{bmatrix}
s-m_At & m_Bt-s \\
 m_{AD}(m_Cs+t) & m_{CB}(m_Ds+t)\\
 \end{bmatrix}, \:\:
 U=
 \begin{bmatrix} (b_A-1)t \\
 (m_Ds+t)-b_A(m_Cs+t)\\
 \end{bmatrix}.$$
That the determinant of $M$ is $\sigma(s,t)$ 
 is verified by direct calculation.

Unpacking the relevant definitions, the above substitutions for $x_C,y_A,y_B$ and $y_C$ show that
 the parallelogram 
 $(x_A,y_A,\ldots, x_D,y_D)$  and the elements $s,t$
satisfy the slope equations in Definition~\ref{slope def},
$$\begin{cases} 
(x_{A}-x_{B})s = (m_Ax_A-m_Bx_B+(b_A-1){w})t \\
(x_B-  x_C)
t=(m_Cx_C- m_Bx_B-{w})s,\\
\end{cases}$$
 if and only if $M\begin{bmatrix}
 x_A \\ x_B
 \end{bmatrix} =wU.$ Thus if $p$  is a rectangle of slope $s/t$, this rectangle satisfies the matrix equation in the lemma. 
 
 Conversely, if the matrix equation holds for $x_A,x_B$ and $w$, then it suffices to show that the parallelogram $p$ is a rectangle. 
Indeed, as above, the matrix equation can be rewritten as 
 $(x_{A}-x_{B})s = (y_{A}-y_{B})t$ and  $(x_{B}-x_C)t=(y_{C}-y_B)s.$
 Either $s\ne 0$ or $t \ne 0$. Suppose $s \ne 0$.  
Now 
 $s(x_A-x_{B})(y_{B}-y_{C}) = t(y_{A}-y_B)(y_{B}-y_C) = s(y_{A}-y_B)(y_{C}-y_B),$ so that since $s \ne 0$, we have $(x_{A}-x_{B})(y_{B}-y_{C})=(y_{A}-y_{B})(y_{C}-y_{B}).$
Similarly, if $t \ne 0$, then since $$t(x_{A}-x_{B})(y_{B}-y_{C}) = s(y_{C}-y_{B})(x_{A}-x_{B})
= (y_{C}-y_{B})(y_{A}-y_{B})t,$$ we obtain $(x_{A}-x_{B})(y_{B}-y_{C})=(y_{C}-y_{B})(y_{A}-y_{B}).$ This proves that 
$p$ is a rectangle. 
 \end{proof}

   The next theorem shows that if there is a rectangle at infinity (which is the case if $\k=\R$), then there is a rectangle at infinity having slope orthogonal to the first. A similar theorem has been proved by Schwartz \cite[Theorem 1.3]{Sch}
   in his setting. 
 
  \begin{theorem}\label{three} 
Let $s,t \in \k$ with not both $s$ and $t$ equal to $0$. Then $s/t$ is the slope of a rectangle at infinity in $\P\C$ if and only if $\sigma(s,t)=0$; if and only if  $-t/s$ is the slope of a rectangle at infinity.
\end{theorem}

\begin{proof}
   Suppose that $s/t$ is the slope of a rectangle at infinity $[x_A:y_A:\cdots:x_D:y_D:0]$. By Lemma~\ref{four eq a},    there is a matrix $M$ whose entries depend on $s$ and $t$ with $\det(M) = \sigma(s,t)$ and 
      \begin{eqnarray} \label{hom matrix} M \begin{bmatrix}
 x_A \\ x_B
 \end{bmatrix} &= &  \begin{bmatrix}
 0 \\ 0
 \end{bmatrix}.
 \end{eqnarray}  
 If $x_A=x_B=0$, then all the coordinates of  $[x_A:y_A:\cdots:x_D:y_D:0]$ are $0$ via the equations in the proof of Theorem~\ref{rectangle in}, a contradiction to the fact that this point is in $\P^8$. Therefore, the matrix equation has a nonzero solution, which implies that $\sigma(s,t)=\det(M) = 0$.  Conversely, if $\sigma(s,t) =\det(M)= 0$, then the equation~(\ref{hom matrix}) has a nonzero solution, and hence by Lemma~\ref{four eq a} there is a rectangle at infinity having slope $s/t$.  
 Finally, from the definition of $\sigma$ it is clear that   $\sigma(s,t) =0$ if and only if $\sigma(-t,s)=0$.
 %
 \end{proof}

As with slope, the aspect ratio of a rectangle is made more complicated by  degenerate rectangles, a case we handle similarly to that of slope by using a pair of homogeneous linear equations. 
As with slope, we view aspect ratio as a point in ${\mathbb{P}}^1$.

  \begin{definition}\label{ar def} A rectangle
  $[x_A:y_A:\cdots:x_D:y_D:w] \in \P\C$ 
   has {\it aspect ratio} $u/v \in \P^1$ if $u,v$ is a solution to the system of equations 
\begin{eqnarray}  \label{ar eqs} \begin{cases} 
(x_B-x_C)U- (y_A-y_B)V  = 0 \\
(y_B-y_C)U +(x_A-x_B)V = 0\\
\end{cases} 
\end{eqnarray} 
 
\end{definition}

The aspect ratio of a rectangle in $\P\C$ is well defined because the orthogonality condition in Definition~\ref{rectangle def} guarantees the system (\ref{ar eqs}) has a nonzero solution. It is a simple consequence of the equations in Definition~\ref{ar def} that in the case where $\k = \R$, 
the absolute value of the aspect ratio   in Definition~\ref{ar def} coincides with the usual definition of aspect ratio of a rectangle in terms of a ratio of  lengths of sides.   

The two vertices $(x_A,y_A)$ and $(x_B,y_B)$ are the same if and only if the aspect  ratio of the rectangle is $0/1$. Similarly,  $(x_B,y_B) = (x_C,y_C)$ if and only if  the aspect ratio is ``infinite;''  i.e., $1/0$. The  (degenerate) rectangle with slope $0/1$ lies on the diagonal $E$ while the rectangle with slope $1/0$  lies on $F$.

\begin{notation} \label{basic poly b 2} Define a  polynomial in $\k[U,V]$     by  \begin{eqnarray*}
\alpha(U,V) & = &  m_{BC}m_{AD}U^2 -\gamma UV+ m_{AB}m_{CD}V^2,
\end{eqnarray*} 
where $\gamma = (m_Am_C-1)(m_B+m_D)-(m_Bm_D-1)(m_A+m_C)$
%
%
 is as in Notation~\ref{basic poly 2}. 

\end{notation}





  
 The analogue of Lemma~\ref{four eq a} for aspect ratio is the following  lemma. A related version of the lemma appears in \cite[Section 2.4]{Sch} and is used for similar purposes of finding the aspect ratios that do not occur for a rectangle in $\C$.

\begin{lemma}   \label{four eq b} 
Let $u,v\in\k$, where not both $u,v$ are zero. There  is a $2\times 1$ vector $U$     and a $2\times 2$ matrix $M$
with
 entries in $\k$ and determinant $m_{CD}\alpha(u,v)$  such that for each $w \in \k$,  a parallelogram 
 $[x_A:y_A:\cdots:x_D:y_D:w] \in \P\C$
  is a rectangle with slope  $u/v$   if and only if 
 $$M  \begin{bmatrix}
 x_A \\ x_B
 \end{bmatrix} = wU.$$  %
 %

\end{lemma}

%


\begin{proof} Let $p=[x_A:y_A:\cdots:x_D:y_D:w]$ be a parallelogram in $\P\C$.  As in the proof of Theorem~\ref{rectangle in} we have   
 $$x_C = \frac{1}{m_{DC}} (m_{AD}x_A+ 
m_{DB}x_B+(b_A-1)w), \:\: x_D = x_A+x_B-x_C
$$   and $y_L = m_Lx_L+b_Lw$ for all $L \in\{A,B,C,D\}$. 
Let 
$$M=\begin{bmatrix}
m_{DA}u+m_Am_{CD}v & m_{BC}u+m_Bm_{DC}v \\
m_Cm_{DA}u+m_{DC}v
& m_Dm_{BC}u+m_{CD}v
\end{bmatrix}, \: \: 
U=\begin{bmatrix}
(b_A-1)(vm_{CD}-u) \\
(b_Am_C-m_D)u \\
 \end{bmatrix}.$$
A calculation shows that the determinant of $M$ is $m_{CD}\alpha(u,v)$.  

Substitution into the defining equations for  aspect ratio show that $u,v,x_A,x_B,w$ satisfy the equations   
 $$
  \begin{cases} 
 0 \: = \: \left(m_{DC}x_B-{m_{AD}}x_A-m_{DB} x_B-(b_A-1)w\right)u  -m_{DC}(m_Ax_A-m_Bx_B+(b_A-1)w)v
  \\ 0 \:= \: \left(m_C\left(m_{AD} x_A+m_{DB} x_B+(b_A-1) {w}\right)-m_{DC}m_Bx_B-m_{DC}w\right)u -m_{DC}x_{AB}v \\
  \end{cases}$$ 
if and only if 
 $M \begin{bmatrix}
x_A \\
x_B \\
 \end{bmatrix} = wU$. Therefore, if  $p$ is a rectangle of aspect ratio $u/v$, then the matrix equation in the lemma holds for $x_A,x_B$ and $w$.  
 
To prove the converse, suppose that the matrix equation holds for $x_A,x_B,w$.  To see that $p$ is a rectangle, 
 we use the fact that either $u \ne 0$ or $ v \ne 0$.   
Suppose $v = 0$. Then $u \ne 0$. The matrix equation implies that the  equations for aspect ratio in Definition~\ref{ar def} are valid for $u$ and $v$.  
Since $v=0$ and $u \ne 0$, we conclude that $x_B = x_C, y_B=y_C,x_A=x_D,y_A=y_D$.  In this case, $p$ is a degenerate rectangle   in $\C$ with aspect ratio  $1/0$.  
Similarly, if $u = 0$, then $v \ne 0$ and $y_A=y_B, x_A=x_B,y_C=y_D,x_C=x_D$, and $p$ is a degenerate rectangle   in $\C$ with aspect ratio $0/1$.  
Finally, if $u,v \ne 0$, then $(x_{A}-x_{B})(y_{B}-y_{C})vu = (y_{C}-y_{B})(y_{A}-y_{B})uv,$ which implies that 
$(x_{A}-x_{B})(y_{B}-y_{C}) + (y_{B}-y_{C})(y_{A}-y_{B})=0.$
Thus $p$ is a rectangle in $\C$ with aspect ratio $u/v$.  
\end{proof}

 \begin{theorem} \label{ar at infinity thm} An element $u/v$ in $\P^1$ is the aspect ratio of a rectangle  at infinity if and only if 
 $\alpha(u,v)=0$.

%
 %
 \end{theorem}

 \begin{proof} The proof   is similar to that of Theorem~\ref{three} but appeals to Lemma~\ref{four eq b} rather than Lemma~\ref{four eq a}. 
%
 \end{proof} 
 
 Theorem~\ref{three} shows how to find the slope of a second rectangle at infinity given a first. The next remark is the analogue for aspect ratio.

 \begin{remark} \label{cross}
 Three of the lines $A,B,C,D$ are parallel if and only if there are two  (degenerate)
rectangles at infinity, one with aspect ratio $0/1$ and the other with aspect ratio $1/0$. 
This follows from Theorem~\ref{ar at infinity thm} since inspection of the polynomial $\alpha$ shows that three of the lines $A,B,C,D$ are parallel if and only if  $\alpha(U,V) = -\gamma UV$; if and only if $\alpha(1,0)=\alpha(0,1) =0$. 
Otherwise, if no three of the lines are parallel and $u/v$ is  the aspect ratio of a rectangle at infinity, then  $(m_{AB}m_{CD}v)/(m_{BC}m_{AD}u)$
 is also the aspect ratio of a rectangle at infinity.   To see this, observe that 
 $\alpha(m_{AB}m_{CD}V,m_{BC}m_{AD}U) = m_{AB}m_{CD}m_{BC}m_{AD}\alpha(U,V),$ so since $\alpha(u,v) = 0$, then $\alpha(m_{AB}m_{CD}v,m_{BC}m_{AD}u)=0$ and the claim follows. \end{remark}

 
 As Schwartz points out in \cite{Sch},
 if $A$ is not parallel to $B$, then the quantity  $ \frac{m_{BC}m_{AD}}{m_{AB}m_{CD}}$ from Remark~\ref{cross} is  a cross ratio, and so  is an invariant of the pairs $A,C$ and $B,D$ under projective transformations of the plane.

\section{The slope path}

While the ideas in Section 3 give equations that govern the  rectangles in $\C$, they do not directly find rectangles of a specified slope. To remedy this, we develop the notion of a slope path that   gives a regular map from ${\mathbb{P}}^1$ to $\P\C$ so that an element $\sigma$ in  ${\mathbb{P}}^1$ is sent to a rectangle  with slope   $\sigma$. 
In Notation~\ref{slope notation} we propose candidates for a parameterization of the slope path, 
 and in Theorem~\ref{big geo} we prove that these candidates work. 
%
Much of what follows in this and later sections depends on how these polynomials are written  in terms of two polynomials $\E(S,T)$ and $\F(S,T)$ that encode information about the diagonals $E$ and $F$ via Lemma~\ref{Delta 0}. Thus it is not so much the existence of the equations in Notation~\ref{slope notation} that is the issue---versions of  parameterizing polynomials can be found by computational means
using Lemma~\ref{four eq a}---but the {\it form} of the polynomials and how they encode  information about the diagonals. This matters for the ability to give conceptual rather than purely computational proofs of the results in this and later sections.  In this sense, Notation~\ref{slope notation} is one of our main theorems. 

Frst we define the polynomials $\E$ and $\F$. For this, recall from Remark~\ref{diagonal remark} and Lemma~\ref{Delta 0} that it cannot happen that both $e_1=e_2=0$ and $f_1=f_2=0$. 

\begin{notation} 
\label{basic poly nota} With indeterminates $S$ and $T$  for $\k$,  we define 
\begin{itemize}
\item  $\E(S,T) =0$ and $  \F(S,T) = 1$  if  $e_1=e_2=0;$
\item   $  \E(S,T) = e_1S+e_2T$ and $\F(S,T) = f_2S-f_1T$  if  $e_1f_1+e_2f_2\ne 0$. 
\end{itemize}
In all other cases,  
let $ \E(S,T)=1 $ and let  $\F(S,T) $ be $\frac{f_2}{e_1}$ if $e_1\ne 0$  and $-\frac{f_1}{e_2}$ if $e_2 \ne 0$. 
  \end{notation}


%




We define the coordinate polynomials for the slope path using $\E$ and $\F$.

\begin{notation}  \label{slope notation} \: 
\smallskip

\begin{center}
\begin{tabular}{ l }
 $\X_A(S,T)  =  (m_CT-S)\cdot {\c E}+S \cdot {\c F}$ \hspace{.3in}  $ \X_C(S,T)  =  (m_DT-S) \cdot {\c E}$  \\ 
  $\X_B(S,T)  =  (m_DT-S)\cdot {\c E} + S \cdot {\c F}$  \hspace{.3in}  $\X_D(S,T)  =  (m_CT-S) \cdot {\c E}$ \\ 
 ${\c X}(S,T) \:  = \: m_{BC}(S-m_DT) \cdot \E-(T+m_BS) \cdot \F$   \\
 $\Y_L(S,T)  = m_{L}\cdot \X_L+b_L\cdot \X$, where $L \in \{A,B,C,D\}$   \\
\end{tabular} 
\end{center}
\smallskip

\end{notation}

\begin{remark} \label{basic poly} The polynomial
$\sigma(S,T)$ of Notation~\ref{basic poly 2} is the polynomial $\X(S,T)$ with $\E$ and $\F$ replaced by $e_1S+e_2T$ and $f_2S-f_1T$, respectively. In the ring $\k[S,T]$,    $\X(S,T)$ divides 
the polynomial $\sigma(S,T)$, with 
equality if and only if $e_1f_1+e_2f_2 \ne 0$. 
 \end{remark}




The polynomials from Notation~\ref{slope notation} are used to find rectangles with specified slope:

\begin{theorem} \label{big geo}  The regular map  $\pi:{\mathbb{P}}^1 \rightarrow {\mathbb{P}}\C$ defined for all $s/t \in {\mathbb{P}}^1$ by
$$
\pi(s/t) = [\X_A(s,t):\Y_A(s,t):\cdots : \X_D(s,t):\Y_D(s,t):\X(s,t)]$$ sends $s/t$ to a rectangle  with slope $s/t$ and aspect ratio $m_{CD}
{\E(s,t)}/{\F(s,t)}$.

\end{theorem}

\begin{proof}  
Define $\E^*(S,T) = e_1S+e_2T$ and $\F^*(S,T) = f_2S-f_1T.$
For each $L \in \{A,B,C,D\}$, let $\X_L^*$ and $\Y^*_L$ be the polynomials $\X_L$ and $\Y_L$ with   $\E$ and $\F$ replaced by  $\E^*$ and $\F^*$. 
A calculation involving the $e_i$ and $f_j$ shows that 
$$b_A\X^*(S, T)= m_{AD}(S-m_CT)\E^*(S, T)+(m_AS+T)\F^*(S, T).$$
With this observation, the following identities are easily verified. 
\begin{center}
\begin{tabular}{ l l l } 
$ \Y^*_A -\Y^*_B \: = \:   m_{CD}S\cdot {\c E}^*$ & \:\:\:\: & 
$ \X^*_A-\X^*_B  \: =  \:  m_{CD}T \cdot {\c E}^*$ \\
$\Y^*_B- \Y^*_C  \: = \:   -T \cdot   {\c F}^* $&  \:\:\:\: & 
$\X^*_B-\X^*_C  \: =  \: S  \cdot {\c F}^*.$  \\
 \end{tabular}
 \end{center}
 We claim that these identities remain true with the asterisk removed from the superscripts.  Let $(\dagger)$ denote these same identities with the asterisk removed. To verify ($\dagger$), consider the cases in the definition of $\E$ and $\F$. 
If $e_1=e_2=0$, then
$$\E(S,T) = \E^*(S,T) = 0, \:\: \F(S,T) = \frac{\F^*(S,T) }{\F^*(S,T)},$$ so $(\dagger)$ holds. If $e_1f_1+e_2f_2 \ne 0$, then $\E^*=\E$ and $\F^*=\F$, so $(\dagger)$ is clear. Finally, if $e_1f_1+e_2f_2 = 0$ and $e_1 \ne 0$ or $e_2\ne 0$, then $$\E(S,T) = \frac{\E^*(S,T)}{\E^*(S,T)}=1, \:\: \F(S,T) = \frac{\F^*(S,T)}{\E^*(S,T)},$$ and so ($\dagger$) is verified. 
 
  Now we prove the theorem.  First, observe there is no common zero of $\X_A,\X_B,\X_C,\X_D$ and $\X$ in $\P^1$.  
  Indeed if $s/t \in \P^1$ is a common zero of these five polynomials, then using ($\dagger$), we conclude that $\E(s,t)=\F(s,t)=0$, which is impossible by the definition of   $\E$ and $\F$ and the fact that it cannot happen that $e_1=e_2=f_1=f_2=0$. Thus  
   the image of $\pi$ is indeed in $\P\C$.

Next, $\X_A + \X_C = \X_B + \X_D$ and  $\Y_A+\Y_C= \Y_B + \Y_D.$ 
  Thus, for all $s/t \in {\mathbb{P}}^1$, 
  $\pi(s/t)$ is a parallelogram   in $\P{\bf C}$.  
  Moreover, by ($\dagger$), 
$$(\X_A-\X_B)(\X_B-\X_C) =  m_{CD}ST{\c E}{\c F} = 
 -(\Y_A-\Y_B)(\Y_B-\Y_C),$$   and so 
$\pi(s/t)$ is a rectangle. 
 That the slope of the rectangle $\pi(s/t)$ is $s/t$ follows from the observation via ($\dagger$) that 
 $$(\X_B(s,t)-\X_A(s,t))s  \: = \: (\Y_B(s,t)-\Y_A(s,t))t,$$
$$
(\Y_C(s,t)-\Y_B(s,t))s \: = \: -(\X_C(s,t) - \X_B(s,t))t.$$ 
The aspect ratio of the rectangle   is calculated similarly from ($\dagger$):
 $$(\Y_B-\Y_A)\F\: =  \: -m_{CD}S\F \: = \: (\X_B-\X_C)m_{CD}\E,$$
$$
(\Y_C-\Y_B)m_{CD}\E \: = \: Tm_{CD}\F\E \: = \: (\X_A - \X_B)\F.$$  Thus the aspect ratio of the rectangle  $\pi(s/t)$ is  $m_{CD}\E(s,t)/\F(s,t)$. 
This proves the theorem.   
%
%
%
%
\end{proof}

\begin{definition} The map $\pi$ is the {\it slope path} for  ${\bf C}$. When no confusion can arise, we refer also to the image of $\pi$ as the slope path for ${\bf C}$. 
\end{definition}

The next corollary gives for all choices  $\sigma \in {\mathbb{P}}^1$ that are not zeros of $\X(S,T)$ the  coordinates of the vertices of the unique inscribed rectangle with slope $\sigma$.

\begin{corollary} \label{count} 
If $s/t \in {\mathbb{P}}^1$ such that $\X(s,t) \ne 0$, then
%
 there is a  rectangle in ${\bf C}$ with slope $s/t$ and vertices  
 \begin{center} 
 $\displaystyle{\left(\frac{\X_L(s,t)}{\X(s,t)},\frac{\Y_L(s,t)}{\X(s,t)}\right) \in L}$, where $L \in \{A,B,C,D\}$.  
 \end{center}


\end{corollary} 

\begin{proof} 
Apply Theorem~\ref{big geo}. 
 \end{proof}

  It can happen that the slope path is the line at infinity. This occurs if and only if $\C$ has twin pairs; 
  see Corollary~\ref{ap}. 
  The significance of this is that if the slope path is the line at infinity, then  there is only one slope possible for the  rectangles   in ${\bf C}$ (see Theorem~\ref{deg char}).


 

 

 \section{Aspect  path}

  We develop the aspect path along similar lines  by first proposing in Notation~\ref{ar notation} the parameterizing polynomials that are needed to find for each $\alpha \in {\mathbb{P}}^1$ a rectangle  in $\P\C$ with aspect ratio $\alpha$. 

\begin{notation}\label{MN} Define polynomials ${\c M}$ and ${\c N}$ in $\k[U,V]$ by  
\begin{itemize}
\item  $\M(U,V) =0$ and $  \N(U,V) = 1$  if  $f_1=e_2=0;$
\item   $  \M(U,V) = \frac{f_1}{m_{CD}}U+e_2V$ and $\N(U,V) = \frac{f_2}{m_{CD}}U- e_1V $  if  $e_1f_1+e_2f_2\ne 0$. 
\end{itemize}
In all other cases,  
let $ \M(U,V)=1 $ and let  $\N(U,V) $ be $-\frac{e_1}{e_2}$ if $e_2\ne 0$  and $\frac{f_2}{f_1}$ if $f_1 \ne 0$. 
\end{notation}

 %


 

The coordinate polynomials for the aspect path are now defined using ${\c M}$ and ${\c N}$.

\begin{notation} \label{ar notation} $\:$

 
\begin{center}
\begin{tabular}{ l }
${\c P}_A(U,V)  =  
 (U-m_{CD}V)\cdot \M -m_CU\cdot \N$ \hspace{.2in}  ${\c P}_C(U,V)  =  U \cdot \M- m_DU\cdot \N$ \\
$ {\c P}_B(U,V)  =  (U-m_{CD}V)\cdot \M -m_DU \cdot \N$  \hspace{.2in}  ${\c P}_D(U,V)  =  U \cdot \M - m_CU\cdot \N.$ \\
${\c P}(U,V)\: = \: \:  (m_{CD}m_BV-m_{BC}U)\cdot \M + (m_{BC}m_DU+m_{CD}V)\cdot \N$ \\
$\Q_L(U,V) = m_L {\c P}_L(U,V)+b_L {\c P}(U,V)$, where $L \in \{A,B,C,D\}$.  
\end{tabular}
\end{center}

\end{notation}

  \begin{remark} \label{basic poly b} 
In the ring $\k[U,V]$,   ${\c P}(U,V)$ divides the polynomial
$m_{CD}\alpha(U,V)$ from Notation~\ref{basic poly b 2},
 with equality holding if and only  if $e_1f_1+e_2f_2\ne 0$.  
  \end{remark}

\begin{theorem} \label{big geo 2}  The regular map  $\phi:{\mathbb{P}}^1 \rightarrow {\mathbb{P}}\C$ defined for all $u/v \in \P^1$ by $$\phi(u/v) =   [{\c P}_A(u,v):{\c Q}_A(u,v): \cdots : {\c P}_D(u,v):{\c Q}_D(u,v):{\c P}(u,v)]$$ sends $u/v$ to a rectangle in $\P\C$   with aspect ratio $u/v$ and slope $\M(u,v)/\N(u,v).$






\end{theorem} 

\begin{proof}  
The proof is similar to that of Theorem~\ref{big geo 2}. First, a calculation shows that 
$$ b_A{\c P}(U,V) \: = \: (m_{CD}m_AV-m_{AD}U)\M  + (m_{AD}m_CU +m_{CD}V)\N. $$
As in the proof of Theorem~\ref{big geo}, we use the following identities, which are verified along similar lines as the identities in the proof of that theorem.
$$
{\c P}_A -{\c P}_B =   m_{DC} U \N, \: 
 \Q_A-\Q_B  =   m_{DC}U\M, \: 
 \Q_B - \Q_C  = m_{CD}V\N, \:
{\c P}_B-{\c P}_C   =  m_{DC}V\M.$$
We first claim  
there is no $u/v \in \P^1$ such that ${\c P}_A(u,v)=\cdots={\c P}_D(u,v)={\c P}(u,v)=0$.  For if there were such a $u/v \in \P^1$, then 
 the fact that $m_C \ne m_D$ and  the above identities imply that ${\c M}(u,v) = {\c N}(u,v) = 0$, which is impossible by the definition of  ${\c M}$ and ${\c N}$.   
Thus the map $\phi$ sends each element of $\P^1$ into $\P \C$.
The proof now proceeds along the same lines as 
 that of Theorem~\ref{big geo}, making use of the stated identities.\end{proof}

\begin{definition} The 
 map $\phi$ (or its image, if the context is clear) is the {\it aspect path} in $\P\C$.
\end{definition}

The aspect path can be the line at infinity for the plane of parallelograms in $\P\C$; see Corollary~\ref{ap}.
 Schwartz \cite[Lemma 2.1]{Sch} has given a version of the aspect path for the case in which $\k =\R$,  ${\bf C}$ is nondegenerate (see Section 7 for a definition) and none of the lines $A,B,C,D$ are parallel or perpendicular. In this case with $\k=\R$, the aspect path has an interesting interpretation, one that is missing from our more general setting: The aspect path is the restriction of a hyperbolic isometry from the hyperbolic plane to the convex domain bounded by the hyperbola that is the rectangle locus for the configuration, where this domain is equipped with the Hilbert metric. 
 
 While Schwartz does not explicitly formulate a slope path, his arguments in  the proof of \cite[Lemma 2.9]{Sch} show that 
 under these same assumptions, there is a homography of the projective line that gives the slope path from the aspect path by factoring through this homography. That the configuration is non-degenerate is needed for the existence of the homography, and so the slope path, like the aspect path,  is only directly deducible in the case $\k=\R$ in \cite{Sch} for the  case in which $\C$ is non-degenerate. 
 Such a homography appears for us too, although for different reasons; see Theorem~\ref{homography}.

\section{Degeneracy of the slope and aspect paths}

The  slope   and aspect paths, since they consist of rectangles in $\P\C$, 
lie on the plane curve of degree $2$ from Theorem~\ref{planar} that is comprised of the rectangles 
in $\P\C$. We show in Corollary~\ref{union} that this curve  is the union of the two paths. But first in Theorem~\ref{deg char} we give criteria for when the aspect and slope paths are lines. The criterion in (1) involving the orthogonality of the diagonals $E$ and $F$ of the configuration (see Section 2) is important for how it connects degenerate behavior of the slope and aspect paths to an elementary geometric property of the configuration. In  the case where $\k = {\mathbb{R}}$ and no two lines in the configuration are parallel or perpendicular, 
Schwartz in \cite[Theorem~3.3]{Sch}  gave a 
 partial version of this criterion by showing in his setting that 
  statement (1)  implies that there are two sets of rectangles  in ${\bf C}$, one set of which consists of rectangles with the same slope and the other of  rectangles with the same aspect ratio.  










\begin{theorem} \label{deg char} The following are equivalent. 
\begin{enumerate}



\item The diagonals $E$ and $F$  are orthogonal. 

\item  
The slope path is a line. 

\item The aspect path is a line. 

\item Every rectangle  on the slope path has the same aspect ratio.
\item Every rectangle  on the aspect path has the same slope (which is the slope $F$).

\item The set of rectangles in $\P\C$  is the union of two distinct lines, one of which is the slope path and the other the aspect path.

\item $e_1f_1+e_2f_2 =0$.











\end{enumerate}
\end{theorem}

\begin{proof}
$(1) \Leftrightarrow (7)$: If $A =B$, $A = D$ or $F$ is the line at infinity, then the diagonals $E$ and $F$ are orthogonal by definition, and so Lemma~\ref{Delta 0} implies that  either $e_1 = e_2 =0$  or $f_1 = f_2=0$; hence  (1) and (7) both hold. 
In all other cases, if  $e_1f_1+e_2f_2=0$, then since by Lemma~\ref{Delta 0} the slopes of $E$ and $F$ are $ e_1/e_2$ and $f_1/f_2$, respectively, the diagonals $E$ and $F$ are orthogonal, and conversely, 
 if $E$ and $F$ are  orthogonal, then  $e_1f_1+e_2f_2 =0$  by Lemma~\ref{Delta 0}.

(2) $\Leftrightarrow$ (3) $\Leftrightarrow$ (7):
The polynomials ${\c E}, {\c F}, {\c M}, {\c N}$ are constant if and only if $e_1f_1+e_2f_2=0$.  Thus if $e_1f_1+e_2f_2=0$, the slope and aspect paths are lines. Conversely, $\E$ and $\F$ are constant if and only if
the polynomials ${\c X}_L$, $L \in \{A,B,C,D\}$, are either $0$ or homogeneous of  degree $1$. Also, $\E$ and $\F$ have degree $1$ if and only if the $\X_L$ are homogeneous of degree $2$.  If  the slope path is a line,  the polynomials $\X_L$ are linear,  
 which in turn implies 
 ${\c E}$ and $ {\c F}$ are constant. Similarly, if the aspect path is a line, then ${\c M}$ and ${\c N}$ are constant. 
 
 (2) $\Rightarrow$ (4): If the slope path is a line, then  ${\c E}$ and  ${\c F}$ must be constant, so  Theorem~\ref{big geo} implies that every rectangle on the slope path has the same aspect ratio.


(4) $\Rightarrow$ (7): 
Suppose every rectangle  on the slope path has the same aspect ratio $u/v$. If $f_1=f_2=0$, then clearly $e_1f_1+e_2f_2=0$, so suppose at least one of $f_1,f_2$ is nonzero. Suppose by way of contradiction that $e_1f_1 +e_2f_2 \ne 0$. Then 
${\c E}(1,0) = e_1$, ${\c F}(1,0)=f_2$,  ${\c E}(0,1) = e_2$ and ${\c F}(0,1)=-f_1$. 
By Theorem~\ref{big geo}, $vm_{CD}\E(\sigma) =u \F(\sigma)$ for all $\sigma \in \P^1$. Thus $vm_{CD} e_1=uf_2$ and $vm_{CD} e_2=-uf_1$, and so $vm_{CD}e_1f_1 = uf_1f_2$ and $vm_{CD}e_2f_2=-uf_1f_2$.   
Since $e_1f_1+e_2f_2 \ne 0$  and $m_C \ne m_D$, it follows that $v =0$.  But then $u \ne 0$ since $u/v \in \P^1$, and so
from the equations $vm_{CD} e_1=uf_2$ and $vm_{CD} e_2=-uf_1$ we conclude that 
 $f_1=f_2 =0$, contrary to assumption. Thus $e_1f_1+e_2f_2 =0$.

  (2) $\Leftrightarrow$  (6):  That (6) implies (2) is clear. Conversely, assume (2). We have established that (2) is equivalent to 
  (3), so both the slope path and the aspect path are lines. These two lines are contained in the set of rectangles in $\P\C$, which is defined by a planar quadric in $\P\C$.  The proof of Theorem~\ref{rectangle in} shows there is an isomorphism $\Gamma$ of varieties from $\P^2 $ onto the plane of parallelograms in $\P\C$.  There are linear polynomials $\ell_1,\ell_2 \in \k[X_A,X_B,X]$ such that
the image of  
  the slope path under $\Gamma^{-1}$ is the zero set of $\ell_1$ in ${\mathbb{P}}^2$ and the image of the  aspect path is the zero set of $\ell_2$ in ${\mathbb{P}}^2$. 
  Let $h$ be the polynomial in the proof of Theorem~\ref{rectangle in}. 
  %
 Let $\overline{\k}$ denote the algebraic closure of $\k$. By the Nullstellensatz, the fact that 
$\ell_1$ and $\ell_2$ are irreducible imply that   
  there are $f,g \in \overline{\k}[X,X_A,X_B]$ such that $h = f\ell_1 = g\ell_2$.  Since $h$ has degree $\leq 2$ and $\ell_1$ and $\ell_2$ have degree $1$, this implies that $h = \lambda \ell_1\ell_2$ for some  $\lambda \in \overline{\k}$.  
  The zero set of $h$ is the union of the zero sets of $\ell_1$ and $\ell_2$. Applying $\Gamma$,  the set of rectangles in $\P\C$ is the union of the slope path and the aspect path.

  Finally, to see that the slope path and the aspect path are not the same, use the fact that we have established already that (2) and (4) are equivalent. Thus, if the slope path is the aspect path, then every rectangle   on the aspect path has 
 the same aspect ratio, contrary to Theorem~\ref{big geo 2}.  From this we conclude that the slope path and aspect path are distinct from one another.

(3) $\Rightarrow$ (5): If the aspect path is a line, then as in the proof that (2) and (3) are equivalent,  
Theorem~\ref{big geo 2} implies that the degrees of ${\c M}$ and  ${\c N}$ are not $1$, and so Theorem~\ref{big geo 2} implies that every rectangle on the aspect path has the same slope.  
To see that this slope is the slope of $F$, 
suppose first that $A = B$.  Then $F= A=B$. 
Since $A = B$, every rectangle in $\C$ has vertices that lie on $A$ and $B$, and so the slope of each rectangle is the slope of $F$.  Now suppose that $A \ne B$.  
 Lemma~\ref{Delta 0} implies that $e_1 \ne 0$ or $e_2 \ne 0$.  
Using Notation~\ref{MN} and 
 Theorem~\ref{big geo 2}, the slope of every rectangle on the aspect path is $-e_2/e_1$, which by Lemma~\ref{Delta 0} is the slope of a line that is orthogonal to  $E$, and hence by the equivalence of (1) and (3) is the slope of $F$.

      (5) $\Rightarrow$ (7):
 The proof is similar to the proof that   (4) implies (7). Suppose every rectangle  on the aspect path has the same slope $s/t$, and suppose by way of contradiction that $e_1f_1+e_2f_2 \ne 0$. Then 
$${\c M}(1,0) =-\frac{f_1}{m_{CD}}, \: \: {\c N}(1,0)=-\frac{f_2}{m_{CD}}, \: \: {\c M}(0,1) = -e_2, \: \: {\c N}(0,1)=e_1.$$
By Theorem~\ref{big geo 2}, $t\M(\sigma) =s \N(\sigma)$ for all $\sigma \in P^1$. Thus $t f_1=sf_2$ and $-t e_2=se_1$, 
and so $te_1f_1 = -se_1f_2$ and $te_2f_2=se_1f_2$.  Since $e_1f_1+e_2f_2 \ne 0$, it follows that $t =0$.  But then $s \ne 0$ since $s/t \in \P^1$, and so $e_1=f_2 =0$, a contradiction to the assumption that $e_1f_1+e_2f_2 \ne 0$. 
Therefore, $e_1f_1+e_2f_2 = 0$.
\end{proof}

  \begin{definition} \label{deg} The configuration ${\bf C}$ is {\it degenerate} if $\C$ satisfies 
  the equivalent statements of Theorem~\ref{deg char}. 
  \end{definition}

  Figure~2 illustrates   the statements in Theorem~\ref{deg char} for  a degenerate configuration.
  If  $\k = \R$, then Theorems~\ref{planar} and~\ref{deg char} imply that
 $\C$ is degenerate if and only if the set of rectangles in $\C$ is a degenerate planar hyperbola or a line.
In this same case $\k = \R$,
 Schwartz \cite[Section 2.2]{Sch} defines  a ``perpendicularity test'' for a configuration, an expression involving the slopes and $y$-intercepts of the lines $A,B,C,D$ that is $0$ if and only if the diagonals $E$ and $F$ are perpendicular. 
 A calculation shows that the expression he gives is $0$ if and only if  $e_1f_1+e_2f_2=0$, and so our interpretation of the $e_i$ and the $f_j$ in Proposition~\ref{Delta 0} in terms of the slopes of  $E$ and $F$ gives another point of view on how his equation encodes the orthogonality of the diagonals. As with our case, Schwartz \cite[Theorem~3.3]{Sch} concludes in his setting that (4), (5) and (6) hold for a degenerate configuration.

\begin{corollary} \label{cases} 
If two lines in $\C$ are equal, $\C$ has twin pairs or $\C$ has dual pairs, then $\C$ is degenerate. 
 \end{corollary}

\begin{proof} This is verified by an easy calculation using 
the fact that $\C$ is degenerate if and only if $e_1f_1+e_2f_2=0$.
\end{proof} 


See Figure~6 for an example in which $A= B$, an example  that covers the case of rectangles inscribed in 3 lines. 
\begin{figure}[h] \label{diagonals non}
 \begin{center}
 \includegraphics[width=0.5\textwidth,scale=.07]{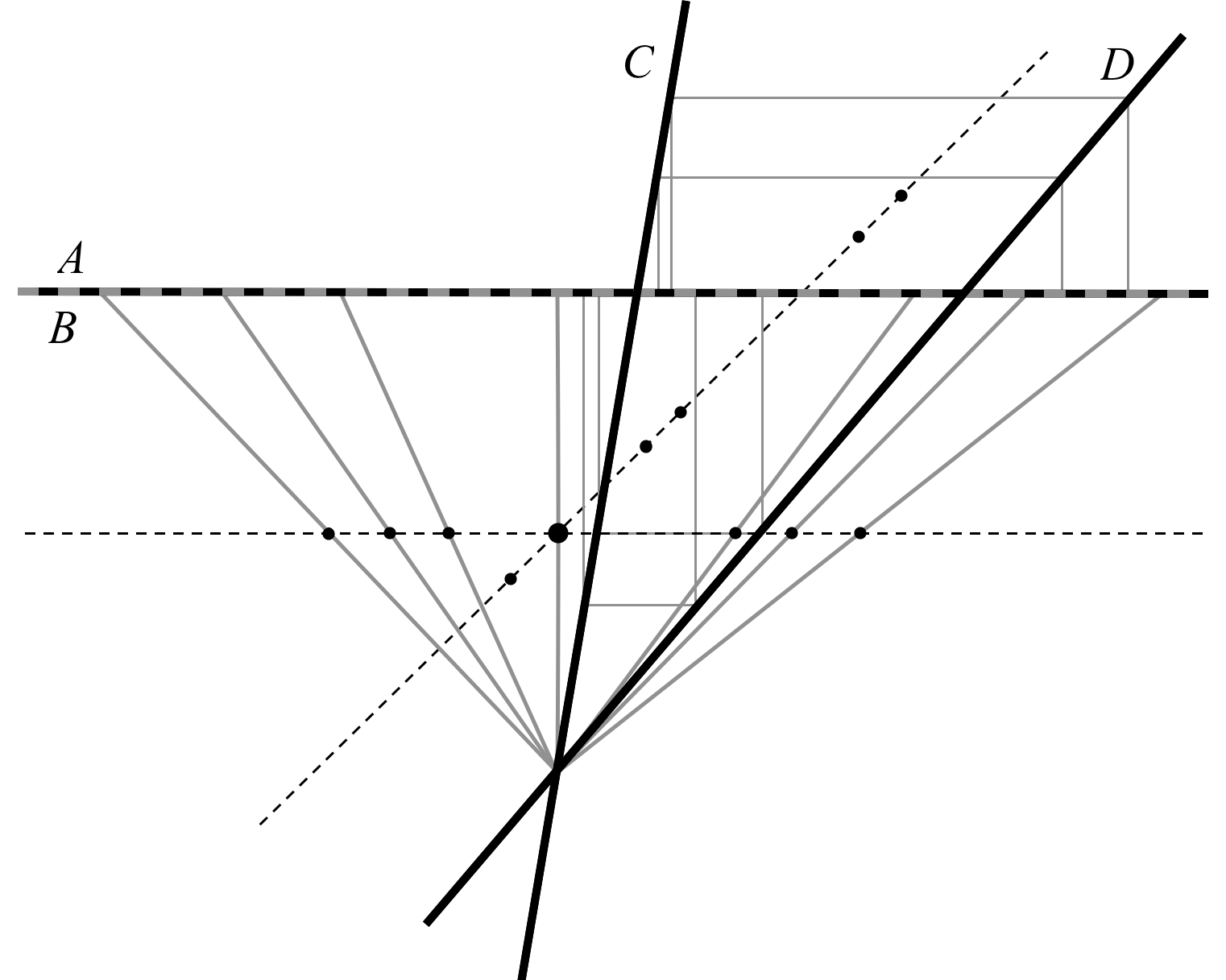} 
 \end{center}
 \caption{Rectangles inscribed in a configuration in which $A=B$. 
The set of centers of the rectangles on the aspect path 
 is the slanted dotted line. All the rectangles  on this line have the same slope but never the same aspect ratio. The set of the centers of the rectangles on the slope path is the horizontal dotted line. The rectangles along this line are all degenerate with aspect ratio $0$ but none have the same slope. The aspect path of centers goes through the midpoint of the altitude of the visible triangle and the midpoint of its base, since these are the two diagonals of the configuration. See Theorem~\ref{Newton}.
  }
\end{figure}
Since degeneracy of $\C$ is equivalent to the orthogonality of the diagonals $E$ and $F$, there are many more cases of degenerate configurations than  Corollary~\ref{cases} might suggest. The next corollary shows that it is always possible to degenerate the configuration $\C$ by  translating one line only, and that  if $\C$ does not have twin pairs, there are only two such translations that work.

\begin{corollary} Let $\k = \R$. If $\C$ does not have twin pairs, then with $m_A,m_B,m_C,m_D$   fixed, there are two choices for $b_A$ in which $\C$ is degenerate, while if $\C$ is does have twin pairs,  any choice of $b_A$ yields a degenerate configuration. 
\end{corollary}

\begin{proof}  By Theorem~\ref{deg char}, $\C$ is degenerate if and only if $e_1f_1+e_2f_2 =0$. Equivalently, using the definition of the $e_i,f_j$, $\C$ is degenerate if and only if $b_A$ is a zero of the polynomial $$m_{BC}(m_Bm_D+1)X^2-\delta X+m_{AD}(m_Am_C+1),$$
where $\delta = (m_Am_C+1)m_{BD}+(m_Bm_D+1)m_{AC}.$
A calculation shows that the discriminant of this polynomial  is 
$$4(m_Am_C-m_Bm_D)^2+(
(m_Am_C+1)(m_B+m_D)-(m_Bm_D+1)(m_A+m_C))^2,$$ which is the discriminant $\Delta$ from the proof of Theorem~\ref{at infinity}. As in that proof, if $\k=\R$, then $\Delta=0$ if and only if $\C$ has twin pairs. As long as $\C$ does not have twin pairs, there are exactly two choices of $b_A$ that result in a degenerate configuration for $\C$.  On the other hand, if $\C$ has twin pairs,  
 Corollary~\ref{cases} implies $\C$ is degenerate. 
\end{proof}


Lemma~\ref{at infinity} shows that an entire  line of rectangles in $\P\C$ may occur at infinity. If this happens, then  $\C$ is degenerate by Theorem~\ref{deg char}. 
In Corollary~\ref{ap} we distinguish when it is the slope path vs.~the aspect path that occurs at infinity. By Theorem~\ref{deg char}(6), the slope and aspect paths cannot both be the line at infinity for the plane of parallelograms in $\P\C$. See Figure 5 for two examples in which the slope path is the line at infinity. 

\begin{corollary} \label{ap} \label{exceptional} The
slope path is the line at infinity for the plane of parallelograms in $\P\C$ if and only if $\C$ has twin pairs, and the aspect path is the line at infinity if and only if $\C$ has dual pairs. 

\end{corollary}

\begin{proof} Suppose that the slope path is the line at infinity. 
By  Lemma~\ref{at infinity},  either $\C$ has twin pairs or dual pairs. Since the slope path is the line at infinity, $\X(S,T) =0$, and so by Remark~\ref{basic poly}, $\sigma(S,T) = 0$. In this case, $m_Am_C = m_Bm_D$, 
so if 
$\C$ has dual pairs, then $m_A \ne 0$ and $m_A=m_B$, so that $m_C = m_D$, contrary to our standing assumption on $C$ and $D$. Therefore, $\C$ has twin pairs. 

 Conversely, if $\C$ has twin pairs, then $\C$ is degenerate by Corollary~\ref{cases} and so 
 the slope and aspect paths are lines by Theorem~\ref{deg char}. By Lemma~\ref{at infinity}, one of these lines is the line at infinity for the plane of parallelograms at infinity. Since $\C$ has twin pairs, examination of the coefficients of the polynomial  $\sigma(S,T)$ shows that $\sigma(S,T)=0$, and so every element of $\P^1$ occurs as the slope of a rectangle at infinity by 
Theorem~\ref{three}.  
 By Theorem~\ref{deg char}, every rectangle on the aspect path has the same slope, so the aspect path is not the line at infinity. Therefore, the slope path is the line at infinity. 
 
 Similarly, if the aspect path is the line at infinity, then ${\c P}(S,T)=0$, and Remark~\ref{basic poly b} implies that   $m_{BC}m_{AD}=0=m_{AB}m_{CD}$. Thus one of the following sets consists of parallel lines: $\{A,C,D\},$ $\{A,B,D\},$ $\{B,C,D\},$ $\{A,B,C\}.$ 
Since   $C$ is not parallel to $D$,  we have either that $m_A=m_B=m_D$ or $m_A=m_B=m_C$. By Lemma~\ref{at infinity}, $\C$ has twin pairs or dual pairs. Since $m_C \ne m_D$, we cannot have that $\C$ has twin pairs along with three slopes being equal, so we conclude $\C$ has dual pairs. 

Conversely, if $\C$ has dual pairs, then $\C$ is degenerate  by Corollary~ref{cases}. A calculation shows  that the coefficients of  $\alpha(U,V)$ are $0$.  By Theorem~\ref{ar at infinity thm}, every element in $\P^1$ occurs as the aspect ratio of a rectangle at infinity, so, since by Theorem~\ref{deg char} every rectangle on the aspect path has the same aspect ratio,  the line at infinity is not the slope path. Therefore, the aspect path is the line at infinity. 
\end{proof}

 Thus if $\k$ is a formally real field, the aspect path is never the line at infinity for the plane of parallelograms, and so there are at most two rectangles at infinity on the aspect path.


 \section{Non-degenerate configuration} 
 
Theorem~\ref{deg char}(7) shows that  a generic choice of the configuration $\C$---that is, a choice of $m_A,m_B,m_C,m_D,b_A$ for which $e_1f_1+e_2f_2 \ne 0$---results in a non-degenerate configuration, and so  the non-degenerate configurations present a  more typical situation. Because of this, it is worthwhile to describe the behavior of the slope and aspect paths in this case too.  
 Negations of the statements in Theorem~\ref{deg char} give useful characterizations of the non-degenerate case, but some of the ideas can be pushed a little farther to obtain stronger statements. We do this in the next theorem.

\begin{theorem} \label{deg cor} The following are equivalent.
\begin{enumerate}
\item ${\bf C}$ is non-degenerate. 
\item Every rectangle in $\C$ lies on the slope path. 
\item Every rectangle in $\C$ lies on the aspect path. 
\item The  slope path is  the  aspect path. 
\item No two rectangles on the aspect path have the same  slope. 
\item No two rectangles on the slope path have the same aspect ratio.
\item No two  rectangles  in $\P\C$  have the  same slope. 
\item No two  rectangles in $\P\C$ have the same aspect ratio. 



\end{enumerate}
\end{theorem}

\begin{proof}
To see that (1) implies (2), suppose that ${\bf C}$ is non-degenerate.
We use the isomorphism $\Gamma$ from the proof of Theorem~\ref{deg char} to work in $\P^2$ instead of $\P\C$.
 By Theorem~\ref{deg char} the slope path is not  a line,
Let $\overline{\k}$ be the algebraic closure of $\k$, let  $g(X,X_A,X_B)$ be the defining equation for the image of the slope path in $\P^2$ under $\Gamma^{-1}$,  and let $h(X,X_A,X_B)$ be the defining equation for the image in $\P^2$ of the set of rectangles in $\P\C$ as in Theorem~\ref{rectangle in}.  Since the image of the slope path lies in this set, the Nullstellensatz implies that $h$ is in the radical of the ideal generated by  $g(X,X_A,X_B)$ in the ring $\overline{\k}[X,X_A,X_B]$.  
Since a projective  rational plane curve that is parameterized by   polynomials has order equal to the highest degree of these polynomials \cite[Exercise~3, p.~151]{Wal}, this implies that $g$ has degree $2$. 
 If $g$ is irreducible in this ring, then since $h$ has degree at most $2$,  $h = \lambda g$ for some $\lambda \in \overline{k}$. In this case, the slope path is the set of rectangles in $\P\C$, from which (2) follows. Otherwise, if $g$ is not irreducible, then $g$ is a product of two linear homogeneous polynomials $\ell_1,\ell_2$ in 
$\overline{\k}[X,X_A,X_B]$. It follows that $h = \mu \ell_1\ell_2=\mu g$ for some $\mu \in \overline{k}$. Thus $h$ and $g$ have the same zeroes in ${\mathbb{P}}^2$, which  proves that the slope path is the set of rectangles in $\P\C$. For the proof that (1) implies (3), apply the same argument to the aspect path. It follows from this also that (1) implies (4).

That (2) implies (5)  follows from the fact that every rectangle  on the slope path has a different slope by Theorem~\ref{big geo}. Similarly, (3) implies (6) 
  by Theorem~\ref{big geo 2}. Also, that (4) implies (5) follows from Theorem~\ref{big geo}. 
That (5)   and (6) each imply (1) follows from Theorem~\ref{deg char}. 
This proves that (1)--(6) are equivalent.
That (4)  implies (7) and (8) follows from the already established fact that (4) implies (5) and (6). 
Also,  (7)    and (8)  each imply (1)  by  Theorem~\ref{deg char}, so (1)--(8) are equivalent. 
%
%
%
%
\end{proof}

Figure 1 illustrates Theorem~\ref{deg cor}. 
With the theorem, we can show finally  that the slope and aspect paths find all rectangles in $\P\C$.

\begin{corollary} \label{union} Every rectangle in $\P\C$ lies on the slope path or the aspect path. 
\end{corollary}

\begin{proof} If ${\bf C}$ is degenerate,  this follows from Theorem~\ref{deg char}(6), while if ${\bf C}$ is not degenerate,  this follows from Theorem~\ref{deg cor}.
\end{proof}

We show next that if  $\C$ is non-degenerate, 
 the slope and aspect paths are computable from each  other by factoring through a homography, and so 
 the aspect ratio of a rectangle  in $\P\C$ depends entirely on its slope, and vice versa,  slope is determined by aspect ratio. Theorem~\ref{deg char} shows this is only true for the non-degenerate case.

\begin{theorem} \label{homography} If $\C$ is non-degenerate, then 
  the slope path is   the aspect path composed with a homography $\Psi:\P^1\rightarrow \P^1$. The aspect path is the slope path composed with $\Psi^{-1}$.  
  \end{theorem}

\begin{proof} 
 First we show there  is a homography $\Psi:\P^1\rightarrow \P^1$ such that  a rectangle in $\P\C$ has slope $s/t \in \P^1$ if and only if its aspect ratio is $\Psi(s/t)$. 
  Define 
 $\Psi,\Phi:{\mathbb{P}}^1 \rightarrow {\mathbb{P}}^1$ 
 for all $a/b \in P^1$ by 
 $\Psi(a/b) = m_{CD}\E(a,b)/\F(a,b)$ and $\Phi(a/b) = \M(a,b)/\N(a,b).$ 
 Theorem~\ref{deg char} implies that $\C$ is non-degenerate if and only if $e_1f_1+ e_2f_2 \ne 0$, and so $\E(S,T) = e_1S+e_2T$, $\F(S,T)=f_2S-f_1T$,  $  \M(U,V) = \frac{f_1}{m_{CD}}U+e_2V$ and $\N(U,V) = \frac{f_2}{m_{CD}}U- e_1V$. 
A simple calculation shows that $\Phi$ is the inverse of $\Psi$.   
   By Theorem~\ref{big geo}, for each $a/b \in \P^1$, $\Psi(a/b)$ is the aspect ratio of the rectangle on the slope path that has slope $a/b$, and by Theorem~\ref{big geo 2}, $\Phi(a/b)$ is the slope of the rectangle on the aspect path that has aspect ratio $a/b$.

Let $\pi,\phi:\P^1 \rightarrow \P\C$ denote the slope and aspect  paths from Theorems~\ref{big geo} and~\ref{big geo 2}, and let $\Psi$ be the homography from Theorem~\ref{homography}. 
Let $s/t \in \P^1$. By Theorem~\ref{big geo 2} and the fact  that $\Psi^{-1} = \Phi$, we have that  $\phi(\Psi(s/t))$ is a rectangle with slope $\Psi^{-1}(\Psi(s/t)) = s/t$.  By Theorem~\ref{big geo}, $\pi(s/t)$ also is a rectangle in $\P\C$ with slope $s/t$, and so  $\pi(s/t) = \phi(\Psi(s/t))$ by Theorem~\ref{deg cor}.  This shows that 
 $\pi = \phi \circ \Psi$. Since $\Psi$ is a homography, 
  $\phi = \pi \circ \Psi^{-1}$.  
\end{proof}

\section{Rectangle locus}

In this section, because we   work with midpoints, we assume $\k$ is a field of characteristic other than $2$. 
 It is sometimes useful to have a  direct way of representing  rectangles in $\C$ with a  point in $\k^2$ rather than $\k^8$. In \cite{OW} and \cite{Sch}, this is done for the case $\k = \R$ using the {\it rectangle locus} of the configuration, the set of centers of the rectangles in $\C$. If neither of the pairs of lines in the configuration consists of parallel lines, then each point on the rectangle locus uniquely determines a rectangle inscribed in sequence in the lines $A,B,C,D$, and so the rectangle locus gives a good way to track the path of inscribed rectangles through  the configuration.
  It is shown in \cite{OW} and \cite{Sch} that in the case in which $\k = {\mathbb{R}}$ and  neither pair $A,C$ nor $B,D$ consists of parallel or perpendicular lines, the rectangle locus is a hyperbola. 
  See Figures~1 and~2 for examples of this.  
  However, if at least 
 one of the pairs   
consists of parallel lines or both pairs consist of perpendicular lines, then the locus can be a line, a line with a segment missing, or a point; see \cite[Section 4]{OW}.  Figures 7 and~8  illustrate two such cases.



\begin{figure}[h] \label{diagonals non}
 \begin{center}
 \includegraphics[width=0.65\textwidth,scale=.08]{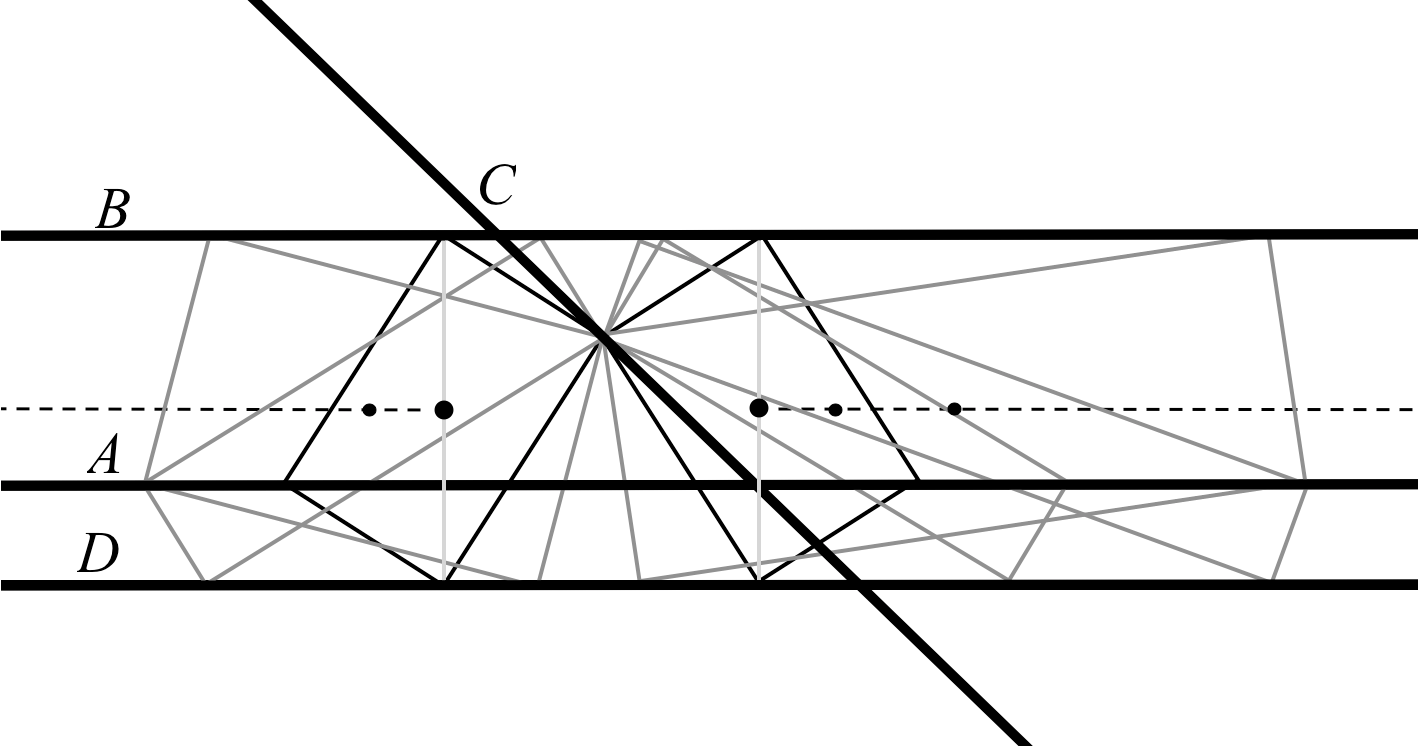} 
 \end{center}
 \caption{Rectangles inscribed in a configuration in which $A, B$ and $D$ are parallel. The rectangle locus, the dotted line in the figure,  is a line with a gap. 
 The configuration is non-degenerate, and so no two rectangles have the same slope or the same aspect ratio (see Theorem~\ref{deg char}). Each point on the locus except the endpoints is the center of two inscribed rectangles. At the endpoints, there is only one inscribed rectangle. This behavior is explained by the fact that the locus is the image of a hyperbola under an affine transformation. 
  }
\end{figure}

\begin{figure}[h] \label{diagonals non}
 \begin{center}
 \includegraphics[width=0.45\textwidth,scale=.07]{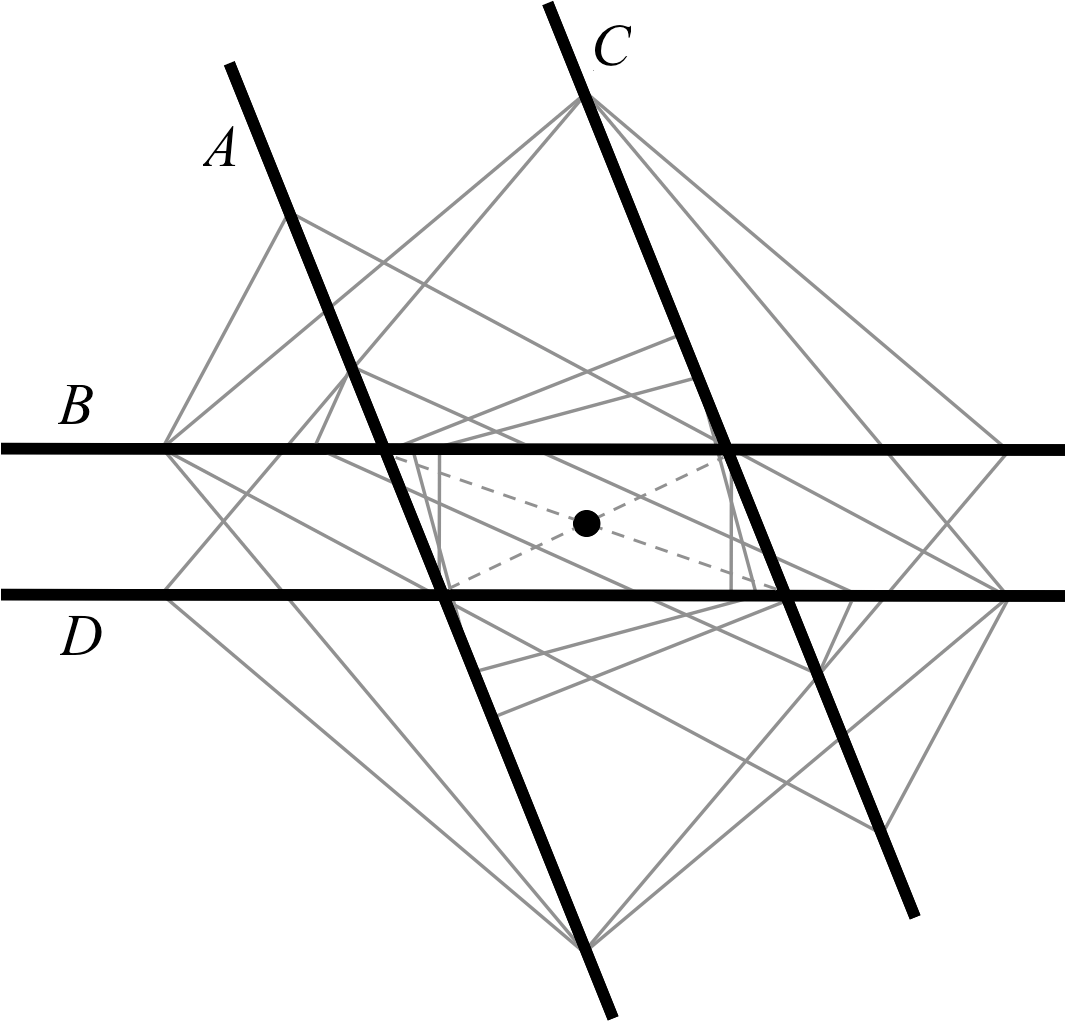} 
 \end{center}
 \caption{Rectangles inscribed in a configuration for which $A \parallel C$ and $B \parallel D$. This configuration is nondegenerate because the diagonals $E$ and $F$ are not orthogonal; see Theorem~\ref{deg char}.  All rectangles share the same center, and so the rectangle locus is a single point. This point is the image of the planar hyperbola of rectangles in $\C$ under a linear transformation. 
  }
\end{figure}

Theorem~\ref{planar} helps explain these observations. This is because 
there is a linear transformation from $\C$ to $\k^2$ that sends a rectangle to its center, and 
 the rectangle locus is the image under this transformation of either a line or a plane curve of degree $2$.  
If $\k = \R$, then by Theorem~\ref{planar} the rectangle locus is the image of either a line or a planar hyperbola, and  if no pairs among $A,B,C,D$ are   parallel or orthogonal, 
 then 
the rectangle locus is a hyperbola. 
(With the assumption of no parallel lines, it is straightforward to see that this linear transformation restricts to a bijection from the set of rectangles in $\C$ to the rectangle locus for $\C$.)


%


For the rest of the section we work under the assumption that none of $A,B,C,D$ are parallel, and we interpret the rectangle locus and rectangles at infinity for the case in which $\C$ is degenerate. 
With $\C$ degenerate, the rectangle locus is by Theorem~\ref{deg char} a pair of lines, one the image of the slope path and the other the image of the aspect path. We say that the image of the slope path is the {\it slope path of centers for $\C$} and the image of the aspect path is the {\it aspect path of centers for $\C$}. The slope path of centers is the image of 
the 
 linear transformation from $\k^8$ to $\k^2$ that sends a rectangle to its center composed with the slope path, and so it sends $s/t \in \P^1$ to the center of a rectangle in $\C$ with slope $s/t$.  A similar statement holds for the aspect path of centers.

 Let $p_{AB}$ denote the intersection of the lines $A$ and $B$, $p_{BC}$ the intersection of $B$ and $C$, etc.. If $\k = \R$, then the midpoints of the  traditional diagonals of the complete quadrilateral formed by $A,B,C,D$ lie on a line, the Gauss-Newton line. It is  straightforward to see that this remains true for any choice of field $\k$ with characteristic other than $2$; i.e., the following points are collinear: $$\frac{1}{2}\left(p_{AB}+p_{CD}\right), \frac{1}{2}\left(p_{AD}+p_{BC}\right), \frac{1}{2}\left(p_{AC}+p_{BD}\right).$$ We refer to the line through these three points as the Gauss-Newton line. Also, 
 for the purpose of stating the next  theorem, we denote by $G$ the third diagonal of the configuration $\C$, that is, the line through the points $p_{AC}$ and $p_{BD}$.   (Recall that $E$ is the line through $p_{AB}$ and $p_{CD}$ and $F$ is the line through $p_{AD}$ and $p_{BC}$.)

\begin{theorem} \label{Newton} Suppose  $\C$ is a degenerate configuration in which no lines are parallel. Then  the aspect  path of centers   is the Gauss-Newton line. If 
at least one of the pairs $A,C$ or $B,D$ is not orthogonal,  the slope path of centers is a line that is parallel to  the diagonal $G$.  
\end{theorem}

\begin{proof}  By Theorem~\ref{deg char}, the slope and aspect paths are lines, so  the slope and aspect paths of centers are lines too. The
points  $p_{AB}$ and $p_{CD}$ are  the vertices of a degenerate rectangle in $\C$ with slope that is orthogonal to the slope of $E$, the line that passes through them.  Similarly, $p_{AD}$ and $p_{BC}$, which lie on $F$, are the vertices of a degenerate  rectangle in $\C$ with slope equal to $m_F$, the slope of $F$. Since  $E$ is orthogonal to $F$ by Theorem~\ref{deg char}, these two degenerate rectangles have slope $m_F$.  

 We claim  these two  rectangles lie on the aspect path. 
  By Theorem~\ref{deg char}, 
 every rectangle on the aspect path, including the rectangle at infinity, has slope $m_F$,  and so by Theorem~\ref{three}, the other rectangle at infinity has slope equal to the slope of $E$.
    By Theorem~\ref{big geo}, no two rectangles on the slope path have the same slope, so there is at most one rectangle on the slope path with slope $m_F$. The 
  point of intersection of the slope path and aspect path is a rectangle having slope $m_F$, so since every rectangle occurs on the slope path or the aspect path by Corollary~\ref{union},  both of these degenerate rectangles lie on the aspect path, and hence their centers, which are
  the midpoints of the degenerate rectangles, lie on the aspect path of centers. Since this path is a line,  it is the Gauss-Newton line.

 To prove the second claim of the theorem, 
suppose that at least one of the pairs  $A,C$ and $B,D$ is not orthogonal. 
By Lemma~\ref{Delta 0} the slope of the diagonal $E$ is $e_1/e_2$.
We have observed already that there are two rectangles at infinity, one having the slope of $E$ and the other the slope of $F$, and that 
the rectangle at infinity for the aspect path has slope equal to the slope of $F$.  Thus the rectangle at infinity for the slope path has slope $e_1/e_2$. 
  By Theorem~\ref{three}, $\X(e_1,e_2)=0$, and 
  by Theorem~\ref{big geo}, the rectangle at infinity on the slope path is $$\pi(e_1/e_2) = 
[\X_A(e_1,e_2):\Y_A(e_1,e_2):\cdots :\X_D(e_1,e_2):\Y_D(e_1,e_2):0].$$
The center of this rectangle is $$\left(
\frac{\X_A(e_1,e_2)+\X_C(e_1,e_2)}{2},
\frac{\Y_A(e_1,e_2)+\Y_C(e_1,e_2)}{2}
\right).$$
Since this rectangle is the point at infinity for the slope path of centers, the slope of the line that is the slope path of centers is $$
\left(\Y_A(e_1,e_2)+\Y_C(e_1,e_2)\right)/
\left(\X_A(e_1,e_2)+\X_C(e_1,e_2)\right).$$
Since $\C$ is degenerate and $A \ne B$, we have $\E(S,T) =1$ and  $\F(S, T) =\frac{f_2}{e_1} $ or $\F(S,T)  = -\frac{f_1}{e_2}$, whichever is defined. Since $e_1f_1+e_2f_2=0$,  regardless of whichever of these fractions is defined, we have $e_1\F(S,T) = f_2$.                    
Calculating, we see 
 \begin{eqnarray*}
 \X_A(e_1, e_2)+\X_C(e_1, e_2) & = & e_2m_C-e_1+e_1\F(S,T)
+ e_2m_D-e_1 \\
&= &  e_2(m_C+m_D)-2e_1+f_2 \\
& = & (b_A-1)(m_C+m_D)-2(b_Am_B-m_A)+  m_{DA}+b_Am_{BC} \\ & = & 
m_{DB}b_A+m_{AC}.
\end{eqnarray*}
Similarly, since $\X(e_1,e_2) =0$, we have 
  \begin{eqnarray*}
 \Y_A(e_1, e_2)+\Y_C(e_1, e_2) & = & 
m_A\X_A(e_1,e_2) + m_C\X_C(e_1,e_2)  \\
& = & m_A(e_2m_C-e_1+e_1\F(S,T))
+m_C( e_2m_D-e_1) \\
& = &  e_2(m_Am_C+m_Cm_D)-(m_A+m_C)e_1+m_Af_2 \\
& = & -b_Am_Bm_C+b_Am_Cm_D+m_Am_D-m_Cm_D \\
& = & m_Cm_{DB}b_A+m_Dm_{AC}
\end{eqnarray*}
The lines $A,C$ and the lines $B,D$ intersect, respectively, in the two points $$\left( \frac{b_A}{m_{CA}}, \frac{m_Cb_A}{m_{CA}}\right),  \:\:
\left( \frac{1}{m_{DB}}, \frac{m_D}{m_{DB}}\right),$$
so the slope of the diagonal $G$ is 
$\left(m_Cb_Am_{DB}+m_Dm_{AC}\right)/\left(b_Am_{DB}+m_{AC}\right),$
which is the same as the slope of the slope path of centers. 
\end{proof}

The restriction in the theorem that at least one of the pairs $A,C$ and $B,D$ is not  orthogonal is necessary. Otherwise,  $\C$ has twin pairs and the slope path is the line at infinity (Corollary~\ref{ap}), and so the slope path is not parallel to $G$.  
 
 We next interpret the rectangles at infinity for a degenerate configuration. To do so, it suffices to describe the slope and aspect ratio of these rectangles since this information uniquely determines the rectangle up to scale. 
 The
 four points of intersection $p_{AB},p_{BC},p_{CD},p_{AD}$  defined above 
are the vertices of a quadrilateral, and it is straightforward to see that the midpoints of  adjacent ``sides'' form the vertices of a parallelogram. (In the case $\k = \R$, this is known as the Varignon parallelogram of the quadrilateral.) 
The center of the parallelogram is the {\it centroid} of $\C$.  
The diagonals $E$ and $F$ are orthogonal if and only if this parallelogram is a rectangle.

 \begin{definition} Suppose $\C$ is a degenerate configuration in which no lines are parallel. The {\it center rectangle} is the rectangle in $\C$ whose center is the intersection of the two lines that comprise the rectangle locus. The {\it centroid rectangle} is the rectangle in $\C$ whose center is the centroid of  $\C$.
 \end{definition} 
 
\begin{figure}[h] \label{diagonals non}
 \begin{center}
 \includegraphics[width=0.9\textwidth,scale=.07]{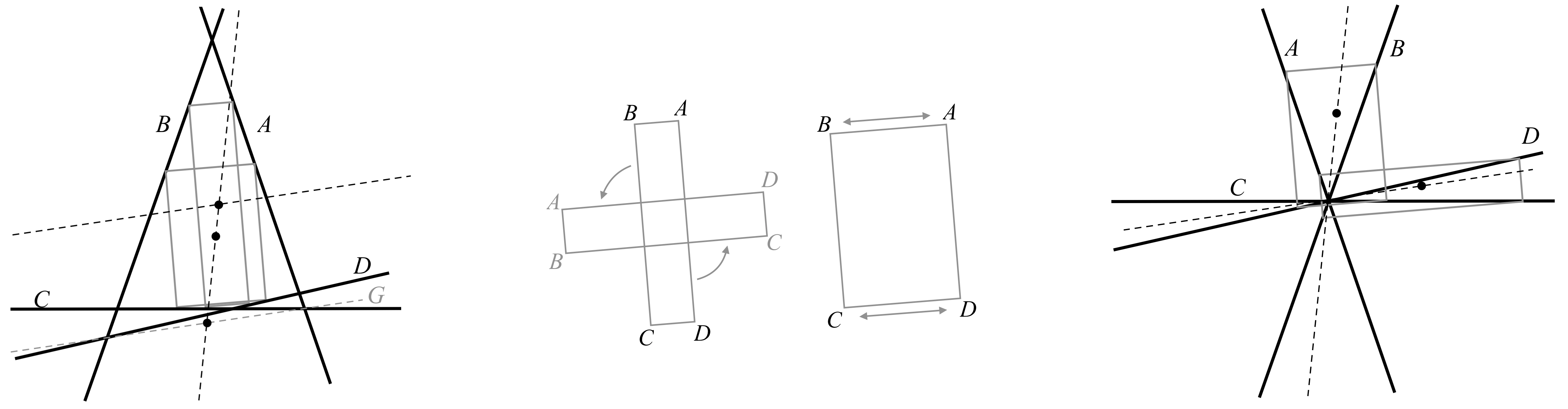} 
 \end{center}
 \caption{The first figure shows the centroid  and center rectangles (wide and narrow, respectively). The rectangles at infinity are shown in the fourth figure. The rectangle at infinity for the slope path is the center rectangle with a quarter turn, while the rectangle at infinity for the aspect path is the centroid rectangle with a flip. 
The middle two images illustrate the turn and flip.  In the the first figure the aspect path of centers is the Gauss-Newton line and the slope path is parallel to the diagonal $G$ (see Theorem~\ref{Newton}). 
  }
\end{figure}

 \begin{theorem} Suppose $\C$ is a degenerate configuration in which no lines are parallel. The rectangle at infinity for the slope path  has  the aspect ratio of the center rectangle and is orthogonal to it, while the rectangle at infinity for the aspect path has the slope of the centroid rectangle  
 and aspect ratio   the negative of the aspect ratio of the centroid rectangle. \end{theorem}
 
 The theorem is more easily stated in the case $\k = \R$: The rectangle at infinity for the slope path is the center rectangle with a quarter turn, while the rectangle at infinity for the aspect path is the centroid rectangle flipped and inscribed in the opposite of  whichever direction, clockwise  or counterclockwise, the centroid rectangle is inscribed.  
 
 \begin{proof} As in the proof of Theorem~\ref{Newton}, 
 the rectangle at infinity for the aspect path has slope $F$ and  the  rectangle at infinity for the slope path has slope equal to the slope of $E$.
 Also by Theorem~\ref{deg char}, every rectangle on the slope  
path has the same aspect ratio, so the rectangle at infinity for the slope path has the same aspect ratio as the center rectangle and is orthogonal to it. 
The aspect ratio of the centroid rectangle is $$(q_{AD}+q_{AB}-q_{AB}-q_{BC})/(p_{AB}+p_{BC}-p_{BC}-p_{CD})=(q_{AD}-q_{BC})/(p_{AB}-p_{CD}).$$ 
Using the calculation of these vertices from 
 the proof of Lemma~\ref{Delta 0}, we have 
 \begin{eqnarray*} q_{AD}-q_{BC} &= & \displaystyle{ \frac{m_Db_Am_{CB}-m_Cm_{DA}}{m_{DA}
 m_{CB}}} \:\:  = \:\:  \displaystyle{-\frac{f_1}{m_{DA}
 m_{CB}}}, \\
 p_{AB}-p_{CD} & =  &\displaystyle{\frac{1-b_A}{m_{AB}}}
=\displaystyle{-\frac{e_2}{m_{AB}}}.
\end{eqnarray*}
 Therefore, the aspect ratio of the centroid is $(m_{AB}f_1)/(m_{DA}m_{CB}e_2).$  
By Theorem~\ref{big geo}, 
the aspect ratio of the rectangle at infinity for the slope path is 
 $(-m_{CD}e_2)/f_1$. By Remark~\ref{cross}, $(-m_{AB}f_1)/(m_{DA}m_{CB}e_2)$ is  the aspect ratio of the rectangle at infinity for the aspect path, and this is the negative of the aspect ratio of the centroid rectangle. 
 \end{proof}


  \begin{figure}[h]  
 \begin{center}
 \includegraphics[width=0.55\textwidth,scale=.07]{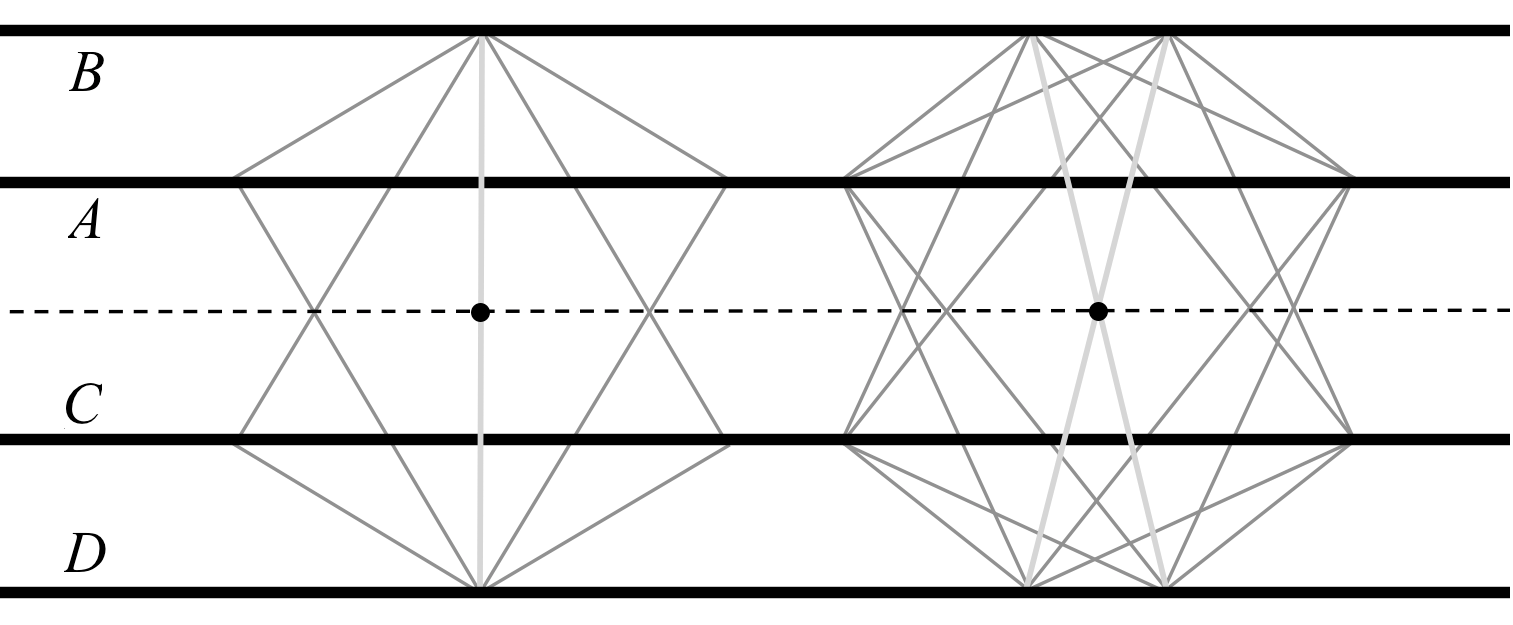} 
 \end{center}
 \caption{Rectangles inscribed in a configuration in which all four lines are parallel.  The rectangles are centered on the dotted line, with infinitely many rectangles centered on each such point. 
  Every point on the locus is the center of 4 rectangles having a fixed diagonal length, except when a diagonal is orthogonal to the four lines, in which case there are only 2 rectangles.
 %
  }
\end{figure}

\begin{remark} \label{four parallel} As discussed in the introduction, our approach throughout the paper assumes that at least two lines in a configuration are not parallel. The situation in which  all four lines $A,B,C,D$ are parallel and the field $\k$ does not have characteristic $2$ can be described as follows. If the pairs $A,C$ and $B,D$ do not share the same midline (i.e., the line consisting of midpoints of the points on either line), then there are no rectangles inscribed in ${\bf C}$. Otherwise, if the pairs  share the same midline, then the midline is the rectangle locus, and 
for each choice of $x_A,x_B \in \k$ with $x_A \ne x_B$, there is a unique rectangle inscribed in ${\bf C}$ whose vertices on  $A$ and $B$ are $(x_A,m_Ax_A+b_A)$ and $(x_B,m_Bx_B+1)$.  
Therefore, the  ``curve'' of rectangles  for this configuration is not a curve at all
but a surface in $\k^8$ parameterized by the points in $\k^2 \setminus \{(x,x):x \in \k\}.$ 
  In any case, this accounts for all the rectangles inscribed in ${\bf C}$. Figure 10 illustrates this situation. 
\end{remark}

\end{document}